\documentclass{elsarticle}
\usepackage{latexsym,hyperref,amsmath,amsthm,
amssymb,tikz}
\usetikzlibrary{arrows,shapes,positioning,decorations.markings, decorations.pathreplacing,decorations.pathmorphing}
\newtheorem{theorem}{Theorem}[section]
\newtheorem{definition}[theorem]{Definition}
\newtheorem{lemma}[theorem]{Lemma}
\newtheorem{conjecture}[theorem]{Conjecture}
\newtheorem{notation}[theorem]{Notation}
\newtheorem{proposition}[theorem]{Proposition}
\newcommand{\Cay}{\mathop{\mathrm{Cay}}}
\newcommand{\Aut}{\mathop{\mathrm{Aut}}}
\newcommand{\Alt}{\mathop{\mathrm{Alt}}}
\newcommand{\PSL}{\mathop{\mathrm{PSL}}}

\begin{document}

\title{Every Finite Non-Solvable Group admits an Oriented Regular Representation}
\author[JMorris]{Joy Morris} 
\address[JMorris]{Department of Mathematics and Computer Science \\
University of Lethbridge \\
Lethbridge, AB. T1K 3M4. Canada}
\ead{joy.morris@uleth.ca}

\author[PSpiga]{Pablo Spiga}
\address[PSpiga]{Pablo Spiga, Dipartimento di Matematica Pura e Applicata,\newline
University of Milano-Bicocca,
Via Cozzi~55, 20126 Milano Italy}  \ead{pablo.spiga@unimib.it}


\begin{abstract}
In this paper we give a partial answer to a 1980 question of Lazslo Babai: ``Which [finite] groups admit an oriented graph as a DRR?" That is, which finite groups admit an oriented regular representation (ORR)? We show that every finite non-solvable group admits an ORR, and provide a tool that may prove useful in showing that some families of finite solvable groups admit ORRs. We also completely characterize all finite groups that can be generated by at most three elements, according to whether or not they admit ORRs.
\end{abstract}
\maketitle

\section{Introduction}
All groups and graphs in this paper are finite. 
Let $G$ be a  group and let $S$ be a subset of $G$. The {\em Cayley digraph} $\Cay(G,S)$ over $G$ with connection set $S$ is the digraph with vertex set $G$ and with $(x,y)$ being an arc if $yx^{-1}\in S$. (In this paper, an {\em arc} is an ordered pair of adjacent vertices.) It is easy to see that the group $G$ acts faithfully as a group of automorphisms of  $\Cay(G,S)$ via the right regular representation. In particular, Cayley digraphs offer a natural way to represent  groups geometrically and combinatorially as groups of automorphisms of digraphs. Clearly, this representation is particularly meaningful if $G$ is the full automorphism group of $\Cay(G,S)$.

In this context it is fairly natural to ask which  groups $G$ admit a subset $S$ with $G$ being the automorphism group of $\Cay(G,S)$; that is, $\Aut(\Cay(G,S))=G$. In this case, we say that $G$ admits a  {\em digraphical regular representation} (or DRR for short). Babai~\cite[Theorem~$2.1$]{babai1} has given a complete classification of the  groups admitting a DRR: except for $$Q_8, \,\,C_2^2,\,\, C_2^3,\,\,   C_2^4\,\ \textrm{and}\,\,  C_3^2,$$ every  group admits a DRR. (Throughout this paper, $Q_8$ denotes the quaternion group of order $8$.)

In light of Babai's result, it is natural to try to combinatorially represent  groups as automorphism groups of special classes of Cayley digraphs. Observe that if $S$ is inverse-closed (that is, $S=\{s^{-1}\mid s\in S\}:=S^{-1}$), then $\Cay(G,S)$ is undirected. Now, we say that $G$ admits a {\em graphical regular representation} (or GRR for short) if there exists an inverse-closed subset $S$ of $G$ with $\Aut(\Cay(G,S))=G$. With a considerable amount of work culminating in \cite{Godsil, Hetzel}, the  groups admitting a GRR have been completely classified. (The pioneer work of Imrich~\cite{
Imrich1,Imrich2,Imrich3} was an important step towards this classification.) It is interesting to observe that, although the classification of the  groups admitting a DRR is much easier than the classification of the  groups admitting a GRR, research and interest first focused on finding GRRs and then on DRRs. It is also worth noting that various researchers have shown that for certain families of groups, almost all Cayley graphs are GRRs, or almost all Cayley digraphs are DRRs~\cite{babai3,Dobson2,Dobson,Godsil}.

We recall that a {\em tournament} is a digraph $\Gamma=(V,A)$ with vertex set $V$ and arc set $A$ such that, for every two distinct vertices $x,y\in V$, exactly one of  $(x,y)$ and $(y,x)$ is in $A$.  After the completion of the classification of DRRs and GRRs,  Babai and Imrich~\cite{babai2} proved that every group of odd order except $C_3 \times C_3$ admits a {\em tournament regular representation} (or TRR for short).  That is, each of these groups $G$ admits a subset $S$ with $\Cay(G,S)$ being a tournament and with $\Aut(\Cay(G,S))=G$. In terms of the connection set $S$, the Cayley digraph $\Cay(G,S)$ is a tournament if and only if $S\cap S^{-1}=\emptyset$ and $G\setminus\{1\}=S\cup S^{-1}$. This observation makes it clear that a Cayley digraph on $G$ cannot be a tournament if $G$ contains an element of order $2$, so only groups of odd order can admit TRRs. 
    
In~\cite[Problem 2.7]{babai1}, Babai observed that there is one class of Cayley digraphs that is rather interesting and that has not been investigated in the context of regular representations; that is, the class of oriented Cayley digraphs (or as Babai called them, oriented Cayley graphs). An {\em oriented Cayley digraph} is in some sense a ``proper" digraph. More formally, it is a Cayley digraph $\Cay(G,S)$ whose connection set $S$ has the property that $S \cap S^{-1}=\emptyset$. Equivalently, in graph-theoretic terms, it is a digraph with no digons. 
\begin{definition}{\rm
The group $G$ admits an {\em oriented regular representation} (or ORR) if there exists a subset $S$ of $G$ with $S \cap S^{-1}=\emptyset$ and $\Aut(\Cay(G,S))=G$. }
\end{definition}
Babai asked in~\cite{babai1} which (finite) groups admit an ORR.
Since a TRR is a special kind of ORR, and $C_3 \times C_3$ is one of the groups that does not admit a DRR (so cannot admit an ORR), the answer to this question for groups of odd order was already known when Babai published his question.

\begin{theorem}\label{obs1}Except for $C_3\times C_3$, every finite group of odd order has an ORR.
\end{theorem}
Since that time, no further progress had heretofore been made in determining which groups admit ORRs. In this paper, we deal with non-solvable groups.

\begin{theorem}[See Theorem~\ref{mainthm}]\label{thrmmain}Every finite non-solvable group admits an ORR. 
\end{theorem}
In a broad sense, the proof of Theorem~\ref{thrmmain} is constructive; that is, given a fixed non-solvable group $G$ and a generating set of minimum cardinality for $G$, by following the proof of Theorem~\ref{thrmmain} together with all of  its subcases, one obtains a subset $S$ of $G$ with $\Cay(G,S)$ an ORR. This can in principle be done for every non-solvable group, but in practice this seems rather difficult.

Moreover, in this paper we provide a tool that may prove useful in future work, to determine families of solvable groups that admit ORRs. 

\begin{theorem}[See Theorem~\ref{JoysLemma}]\label{thrmJoy}
Let $G$ be a finite group that admits a five-product-avoiding generating set $\{a_1, \ldots, a_\ell\}$ with the following properties:
\begin{enumerate}
\item[(i)] $|a_i|>2$ for every $i\in \{1,\ldots,\ell\}$; and
\item[(ii)] $|a_{i+1}a_i^{-1}|>2$ for every $i\in\{1,\ldots,\ell-1\}$.
\end{enumerate}
Then $G$ admits an ORR if and only if $G \not\cong Q_8$, $G \not\cong C_3 \times C_2^3$, and $G \not\cong C_3 \times C_3$. 
\end{theorem}
(We refer to Definition~\ref{5pg} for the concept of a five-product-avoiding generating set. Here we simply observe that every generating set of minimum cardinality, or more generally every irredundant generating set, is five-product-avoiding.)

We also consider all groups that can be generated by at most three elements, and characterize which of these groups admit ORRs.

In Section~\ref{prelim}, we will give some preliminary results and background that will prove useful in the rest of the paper. Section~\ref{2-gen} will examine groups that admit a generating set consisting of at most two elements, and characterize them according to which ones admit ORRs. Section~\ref{3-gen} will provide a similar characterization for groups that admit a generating set consisting of at most three elements. In Section~\ref{JoyLemma} we prove Theorem~\ref{thrmJoy} and some other tools which we believe will be useful in future work on solvable groups. Finally, Section~\ref{non-sol} contains the proof of our main result Theorem~\ref{thrmmain}, using the tools presented in Section~\ref{JoyLemma} and some group-theoretic arguments. The group-theoretic arguments depend upon the Classification of the Finite Simple Groups.

Based on some computer computations and on the work in this paper we dare to make the following conjecture.

\begin{conjecture}\label{conj}
Every finite  group $G$ admits an ORR, unless one of  the following occurs:
\begin{enumerate}
\item[(i)]$G$ is generalized dihedral with $|G|>2$;
\item[(ii)]$G$ is isomorphic to one of the following eleven groups
\begin{align*}
&Q_8,\,C_4\times C_2,\, C_4\times C_2^2,\, C_4\times C_2^3,\, C_4\times C_2^4,\,C_3^2,\,C_3\times C_2^3,\\
&\langle a,b\mid a^4=b^4=(ab)^2=(ab^{-1})^2=1\rangle \text{ (of order $16$)},\\
&\langle a,b,c\mid a^4=b^4=c^4=(ba)^2=(ba^{-1})^2=(bc)^2=(bc^{-1})^2=1,\\
&\qquad\qquad a^2=c^2,a^c=a^{-1}, a^2=b^2\rangle \text{ (of order $16$)},\\
&\langle a,b,c\mid a^4=b^4=c^4=(ab)^2=(ab^{-1})^2=1,\\
&\qquad\qquad(ac)^2=(ac^{-1})^2=(bc)^2=(bc^{-1})^2=a^2b^2c^2=1\rangle \text{ (of order $32$)},\\
& D_4\circ D_4 \text{ (the central product of two dihedral groups of order $8$,} \\
&\qquad\qquad \text{which is the extraspecial group of order $32$ of plus type)}.
\end{align*}
\end{enumerate}
\end{conjecture}
(We say that a group $G$ is {\em generalized dihedral} if $G$ contains an abelian subgroup $A$ with $|G:A|=2$ and an element $\tau\in G\setminus A$ such that $\tau^2=1$ and $a^\tau=a^{-1}$, for every $a\in A$.)


Babai~\cite{babai1} has observed that generalized dihedral groups of order greater than $2$ do not admit ORRs (see Section~\ref{prelim} for a proof of this fact), and hence generalized dihedral groups of order greater than $2$ are genuine exceptions in Conjecture~\ref{conj}. Moreover, a computation with the invaluable help of the computer algebra system \texttt{magma} can be used to prove that the eleven groups listed above also do not admit ORRs. 

Recently, combinatorial representations of  groups has developed  some new vitality and we refer to~\cite{Dobson,XF,MSV,Spiga} for some recent work on similar problems.

\section{Preliminaries}\label{prelim}

We begin with some notation we will require from graph theory.

\begin{notation}{\rm
For a graph $\Gamma$ and a subset $S$ of the vertices of $\Gamma$, $\Gamma[S]$ denotes the induced subgraph of $\Gamma$ on the vertices of $S$. }
\end{notation}

Now we give some group-theoretic notation.

\begin{notation} {\rm Let $G$ be a  group.
\begin{itemize}
\item If $G$ acts on a set $\Omega$, and $x \in \Omega$, then $G_x$ denotes the subgroup of $G$ that fixes $x$.
\item We use $d(G)$ to denote the minimum cardinality of a generating set for $G$.
\item A generating set $S$ for $G$ is said to be {\em irredundant} if, for every $s\in S$, the set $S\setminus\{s\}$ is no longer a generating set for $G$. Observe that, every generating set $S$ for $G$  with $|S|=d(G)$ is irredundant.
\item By a slight abuse of terminology, in order to make the notation less cumbersome, when $\{a_1,\ldots,a_\ell\}$ is a generating set for $G$, we sometimes simply say that $a_1,\ldots,a_\ell$ is a generating set for $G$.
\end{itemize}}
\end{notation}

In their work on the GRR problem, Nowitz and Watkins proved a lemma that is very useful in our context also.

\begin{lemma}[Nowitz and Watkins~\cite{NW}]\label{Watkins-Nowitz}
Let $G$ be a group, let $S$ be a subset of $G$, let $\Gamma=\Cay(G,S)$ and let $X$ be a subset of $S$. If $\varphi$ fixes $X$ pointwise for every $\varphi\in \Aut(\Gamma)_1$, then $\varphi$ fixes $\langle X\rangle$ pointwise for every $\varphi\in \Aut(\Gamma)_1$. In particular, $\Aut(\Gamma)_1=1$ if
\begin{itemize}
\item $G=\langle X\rangle$, or 
\item $\Gamma[S]$ is asymmetric.
\end{itemize}
\end{lemma}

If $\Gamma=\Cay(G,S)$ and $\Aut(\Gamma)_1$ is trivial, then $\Aut(\Gamma)= G$ so that $\Gamma$ is a DRR for $G$, and therefore an ORR if the connection set is asymmetric. We will use this fact repeatedly when we cite the above lemma.

The following lemma is a rather obvious observation, but it will be used so often in the sequel that we prefer to highlight it. 
\begin{lemma}\label{lemma:A-1}Let $G$ be a group and let $a,b\in G$ with $|ab|=|ab^{-1}|=2$. Then 
\begin{align*}
bab&=a^{-1},&ba^{-1}b&=a,&aba&=b^{-1},&ab^{-1}a&=b,\\
b^{-1}a^{2}b&=a^{-2},&a^{-1}b^2a&=b^{-2}.
\end{align*}
Also, if either $a$ or $b$ has odd order, then the other has order  $2$.
\end{lemma}
\begin{proof}The first four equalities are clear from the fact that $1=(ba)^2=baba$ and $1=(ba^{-1})^2=ba^{-1}ba^{-1}$. Now, we deduce
$$a^2ba^2=a(aba)a=ab^{-1}a=b$$
and hence $b^{-1}a^{2}b=a^{-2}$. The last equality follows with a similar computation.

Suppose that $|a|$ is odd. The fifth equality yields that $b$ acts by conjugation inverting the elements of $\langle a^2\rangle=\langle a\rangle$, and hence $b^{-1}ab=a^{-1}$. Now, the first equation yields $bab=a^{-1}=b^{-1}ab$, and hence $b^2=1$. As $a$ has odd order and $|ab|=2$, we cannot have $b=1$, so $|b|=2$. Our final claim follows by reversing the roles of $a$ and $b$ in this argument and using the third and sixth equalities.
\end{proof}

Babai also pointed out in~\cite{babai1} that generalized dihedral groups of order greater than $2$ can never admit an ORR.
\begin{definition}{\rm
Let $A$ be an abelian group. The {\em generalized dihedral group} over $A$ is the group $\langle \tau, A \rangle$ with $|\tau|=2$ and $\tau a \tau=a^{-1}$ for every $a \in A$.}
\end{definition}
(See also the first paragraph following Conjecture~\ref{conj}.)
In the special case where $A$ is cyclic, this is the dihedral group over $A$. Observe that, unless $|G|=2$, if $\Cay(G,S)$ is an ORR, then $\Cay(G,S)$ is connected and hence $S$ is a generating for $G$. Now, Babai's observation follows immediately from the fact that if $G$ is the generalized dihedral group over the abelian group $A$, then every element of $G\setminus A$ has order $2$. Thus every generating set $S$ for $G$ must contain an involution, so that $S \cap S^{-1} \neq \emptyset$. This renders understanding generalized dihedral groups very important when we are studying ORRs.

We conclude this section with a slightly more technical result, showing that for every group $G$, as long as $G$ is not generalized dihedral we can always find a generating set $S$ for $G$ with $|S|=d(G)$ and $S\cap S^{-1}=\emptyset$. This implies that a  group $G$ admits a connected oriented Cayley digraph, if and only if $G$ is not a generalized dihedral group. 

\begin{lemma}\label{lemma:A1}
Let $G$ be a  group. Every generating set for $G$ of cardinality $d(G)$ contains at least one involution if and only if $G$ is a generalized dihedral group.
\end{lemma}
\begin{proof}
If $G$ is generalized dihedral, then $\langle g\in G\mid |g|>2\rangle$ is a proper subgroup of $G$, and hence every generating set for $G$ contains at least one involution.

We prove the other implication. Let $\{a_1,\ldots,a_{\ell}\}$ be a generating set for $G$ with $\ell=d(G)$ and as few involutions as possible. Relabelling the index set $\{1,\ldots,\ell\}$ if necessary, we may assume that $a_1$ is an involution. 

Let $j\in \{2,\ldots,\ell\}$. Now, $\{a_1a_j,a_2,a_3,\ldots,a_{\ell}\}$ is still a generating set for $G$ of cardinality $\ell$. Since this generating set cannot contain fewer involutions than the original generating set,  the element $a_1a_j$ must be an involution. Thus
$$1=(a_1a_j)^2=a_1a_ja_1a_j=a_1^2a_j^{a_1}a_j=a_j^{a_1}a_j$$
so $a_j^{a_1}=a_{j}^{-1}$; that is, conjugation by $a_1$ inverts $a_j$. 

Let $i,j\in \{2,\ldots,\ell\}$ with $i\neq j$. Arguing as above, $\{a_1a_ia_j,a_2,a_3,\ldots,a_{\ell}\}$ is still a generating set for $G$ of cardinality $\ell$. Since this generating set cannot contain fewer involutions than the original generating set,  the element $a_1a_ia_j$ must be an involution. Thus
$$1=(a_1a_ia_j)^2=a_1a_ia_ja_1a_ia_j=a_1^2(a_ia_j)^{a_1}a_ia_j=(a_i^{a_1}a_j^{a_1})a_ia_j=(a_i^{-1}a_j^{-1})a_ia_j$$
so $a_ia_j=a_{j}a_i$; that is, $a_i$ and $a_j$ commute. 

This shows that $N=\langle a_2,\ldots,a_\ell\rangle$ is an abelian normal subgroup of $G$. Since $G=\langle N,a_1\rangle$ and $a_1$ has order $2$, we have $|G:N|=2$. Moreover, since the action of $a_1$ by conjugation inverts the generators $a_2,\ldots,a_\ell$, we see that $G$ is a generalized dihedral group.
\end{proof}

\section{Groups with $d(G)\le 2$}\label{2-gen}

Clearly, if $d(G)=0$, then $|G|=1$ and $\Cay(G,\emptyset)$ is an ORR. Similarly, if $G$ is a group with $d(G)=1$, then $G=\langle a\rangle$ is  cyclic and $\Cay(G,\{a\})$ is an ORR, unless $|G|=2$. However, when $|G|=2$, $\Cay(G,\emptyset)$ is an ORR. Next, in this section, we will deal with groups $G$ such that $d(G)=2$. We begin with a structural decomposition for such groups.

\begin{lemma}\label{lemma:A2}Let $G$ be a  group with $d(G)=2$. Then one of the following holds:
\begin{enumerate}
\item[(i)]$G$ is abelian;
\item[(ii)]$G$ is generalized dihedral;
\item[(iii)]$G$ admits a generating set $\{a,b\}$ with $|a|,|b|>2$, $|ba^{-1}|=|ba|=2$ and $[a,b]\ne 1$;
\item[(iv)]$G$ admits a generating set $\{a,b\}$ with $|a|,|b|,|ba^{-1}|>2$ and $[a,b]\ne 1$.
\end{enumerate}
\end{lemma}
\begin{proof}
Assume that $G$ satisfies neither~(i), nor~(ii), nor~(iv). By Lemma~\ref{lemma:A1}, $G$ admits a generating set $\{a,b\}$ with $|a|,|b|>2$; as $G$ is not abelian, $[a,b]\ne 1$. 
Observe that both $\{a,b\}$ and $\{a^{-1},b\}$ are generating sets for $G$. In particular, as $G$ does not satisfy~(iv), we get $|ba^{-1}|=2$ and  $|ba|=|b(a^{-1})^{-1}|=2$, and hence part~(iii) holds.
\end{proof} 

Next we consider the groups that satisfy Lemma~\ref{lemma:A2}~(iii), and determine which of them admit ORRs.

\begin{lemma}\label{lemma:A3}
Let $G$ be a group with $d(G)=2$ and with a generating set $\{a,b\}$ such that $|a|,|b|>2$, $|ba^{-1}|=|ba|=2$ and $[a,b]\ne 1$. Then $\langle a^2,b^2\rangle$ is a non-identity normal abelian subgroup of $G$. Moreover, one of the following holds:
\begin{enumerate}
\item[(i)]$|a|>4$ and  $\Cay(G,\{a,a^2,b\})$ is  an ORR;
\item[(ii)]$|b|>4$ and $\Cay(G,\{a,b,b^2\})$ is an ORR; 
\item[(iii)]$G$ has presentation $\langle a,b\mid a^4=b^4=(ab)^2=(ab^{-1})^2=1\rangle$, $G$ has order $16$, and $G$ admits no ORR.
\end{enumerate}
\end{lemma}

\begin{proof}
Set $N=\langle a^2,b^2\rangle$. Observe that $N\ne 1$ because $a^2\in N$ and $a^2\ne 1$. From Lemma~\ref{lemma:A-1}, we see that $b^{-1}a^2b=a^{-2}$ and $a^{-1}b^2a=b^{-2}$ and hence $N\unlhd \langle a,b\rangle=G$. Moreover, $(a^2)^{b^2}=((a^2)^b)^b=(a^{-2})^b=a^2$ and hence $[a^2,b^2]=1$, that is, $N$ is abelian.

Suppose that $|a|>4$. Let $S=\{a,a^2,b\}$ and $\Gamma=\Cay(G,S)$. Since $|a|>4$, we immediately see that $\Gamma$ is an oriented Cayley digraph. Consider $\Gamma[S]$, which is the induced subgraph of $\Gamma$ on the neighbourhood of the vertex $1$. Observe that $(a,a^2)$ is an arc of $\Gamma[S]$. Using the irredundancy of $a,b$, we see that neither $(a,b)$ nor $(b,a)$ are arcs of $\Gamma[S]$: see Figure~\ref{Fig1}. (This is a tedious but rather straightforward computation, and similar computations will be required repeatedly in this paper, so we will give the details of this one. In fact, if $(a,b)$ is an arc of $\Gamma[S]$, then $ba^{-1}\in S=\{a,a^2,b\}$. Clearly, $ba^{-1}=a$ yields $b=a^2$, similarly $ba^{-1}=a^2$ yields $b=a^3$, and finally $ba^{-1}=b$ yields $a=1$; in all three cases we obtain a contradiction. The argument when $(b,a)$ is an arc of $\Gamma[S]$ is entirely similar.) If $\Gamma[S]$ admits no non-identity automorphism, then $\Gamma$ is an ORR by Lemma~\ref{Watkins-Nowitz}. Suppose then that $\Gamma[S]$ admits a non-identity automorphism. Then $(b,a^2)$ must be an arc of $\Gamma[S]$, and hence $a^2b^{-1}\in S$. By the irredundancy of $a,b$, this happens only if $a^2b^{-1}=b$, that is, $a^2=b^2$. Now,
$$a^{-2}=(a^2)^b=(b^2)^b=b^2=a^2$$
and hence $a^4=1$, but this contradicts $|a|>4$. Thus part~(i) holds.

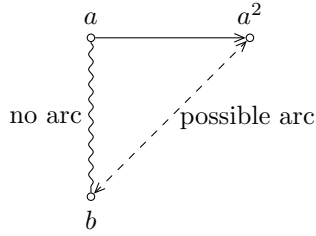
\begin{figure}[!hhh]
\begin{center}
\begin{tikzpicture}[node distance=2cm]
\node[circle,draw,inner sep=1pt, label=90:$a^2$](A0){};
\node[left=of A0,circle,draw,inner sep=1pt, label=90:$a$](A1){};
\node[below=of A1,circle,draw,inner sep=1pt,label=-90:$b$](A2){};
\draw[-angle 45, ] (A1) to  (A0);
\draw[-,decorate,
decoration=
{snake,amplitude=.3mm,segment length=2mm,post length=0mm}
] (A1) to node[left]{no arc} (A2);
\draw[angle 45-angle 45, ,dashed] (A2) to node[right]{possible arc} (A0);
\end{tikzpicture}
\end{center}
\caption{Figure for the proof of Lemma~\ref{lemma:A3}}\label{Fig1}
\end{figure}

 A symmetric argument with $a$ replaced by $b$ yields that $\Cay(G,\{a,b,b^2\})$  is an ORR when $|b|> 4$, and hence part~(ii) holds.

Suppose that $|a|,|b|\le 4$. Since $|a|,|b|>2$, Lemma~\ref{lemma:A-1} implies that $|a|=|b|=4$. 
Then $G$ is a quotient of the group $P=\langle x,y\mid x^4=y^4=(xy)^2=(xy^{-1})^2=1\rangle$. A computation with \texttt{magma} \cite{magma} shows that $P$ has order $16$ and that in each non-abelian proper quotient of $P$ the element $x$ or the element $y$ has order less then $4$. As $|a|=|b|=4$ and $[a,b]\ne 1$, we have $G=P$. Finally, with the invaluable help of \texttt{magma} \cite{magma} we check that the group $G$ admits no ORR. Thus part~(iii)  holds.
\end{proof}

We now consider groups that satisfy either (i) or (iv) but do not satisfy (ii) of Lemma~\ref{lemma:A2}; that is, their minimal generating sets have two elements, and they are either abelian, or admit a generating set $\{a, b\}$ with $|a|, |b|, |ba^{-1}|>2$, but they are not elementary abelian $2$-groups (which are generalized dihedral groups).

\begin{lemma}\label{A:i-or-iv}
Let $G$ be a group with $d(G)=2$ that admits a generating set $\{a,b\}$ with $|a|,|b|>2$.
\begin{enumerate}
\item[(i)] If $G$ is non-abelian and $|ba^{-1}|>2$ then either $G \cong Q_8$, or $\Gamma=\Cay(G,S)$ with $S=\{a,b,ba^{-1}\}$ is an ORR for $G$ and $\Gamma[S]$ is asymmetric.
\item[(ii)] If $G$ is abelian then $G$ admits an ORR unless $G \cong C_3 \times C_3$ or $C_4 \times C_2$.
\end{enumerate}
Moreover, $Q_8$, $C_3\times C_3$ and $C_4\times C_2$ admit no ORR.
\end{lemma}

\begin{proof}
Suppose first that $G$ is not abelian. Let $\Gamma=\Cay(G,S)$, with $S=\{a,b,ba^{-1}\}$; this is an oriented Cayley graph. Observe that in $\Gamma[S]$, there is an arc from $a$ to $b$. Since $\{a,b\}$ is a generating set of minimum cardinality for $G$ and $[a,b] \neq 1$, calculations show that there is no arc from $ba^{-1}$ to $b$ (this would require $bab^{-1} \in S$). Furthermore, calculations show that there is no arc from $a$ to $ba^{-1}$ (this would require $ba^{-2} \in S$). Therefore $\Gamma[S]$ is one of the four graphs shown in Figure~\ref{Fig904}.
\begin{figure}[!hhh]
\begin{center}
\begin{tikzpicture}[node distance=0.85cm]
\node[circle,draw,inner sep=1pt, label=90:$a$](A0){};
\node[right=of A0,circle,draw,inner sep=1pt, label=90:$b$](A1){};
\node[below=of A0,circle,draw,inner sep=1pt,label=-90:$ba^{-1}$](A2){};
\draw[-angle 45, ] (A0) to  (A1);
\node[right=of A1,circle,draw,inner sep=1pt, label=90:$a$](AA0){};
\node[right=of AA0,circle,draw,inner sep=1pt, label=90:$b$](AA1){};
\node[below=of AA0,circle,draw,inner sep=1pt,label=-90:$ba^{-1}$](AA2){};
\draw[-angle 45, ] (AA0) to  (AA1);
\draw[-angle 45, ] (AA1) to  (AA2);
\node[right=of AA1,circle,draw,inner sep=1pt, label=90:$a$](AAA0){};
\node[right=of AAA0,circle,draw,inner sep=1pt, label=90:$b$](AAA1){};
\node[below=of AAA0,circle,draw,inner sep=1pt,label=-90:$ba^{-1}$](AAA2){};
\draw[-angle 45, ] (AAA0) to  (AAA1);
\draw[-angle 45, ] (AAA2) to  (AAA0);
\node[right=of AAA1,circle,draw,inner sep=1pt, label=90:$a$](AAAA0){};
\node[right=of AAAA0,circle,draw,inner sep=1pt, label=90:$b$](AAAA1){};
\node[below=of AAAA0,circle,draw,inner sep=1pt,label=-90:$ba^{-1}$](AAAA2){};
\draw[-angle 45, ] (AAAA0) to  (AAAA1);
\draw[-angle 45, ] (AAAA1) to  (AAAA2);
\draw[-angle 45, ](AAAA2) to (AAAA0);
\end{tikzpicture}
\end{center}
\caption{Figure for the proof of Lemma~\ref{A:i-or-iv}}\label{Fig904}
\end{figure}
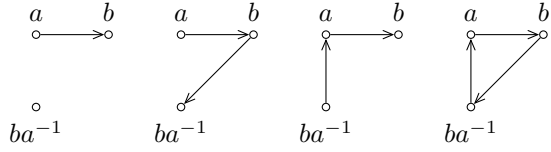

In the first three cases $\Gamma[S]$ is asymmetric and hence, by Lemma~\ref{Watkins-Nowitz}, $\Gamma$ is an ORR for $G$. Suppose then that $\Gamma[S]$ is the fourth graph in Figure~\ref{Fig904}. Calculations show that the arc from $b$ to $ba^{-1}$ exists only if $ba^{-1}b^{-1}=a$ so that $b$ inverts $a$, and the arc from $ba^{-1}$ to $a$ exists only if $a^2b^{-1}=b$, so that $b^2=a^2$. Thus we have $a^2=b^2=b^{-1}b^2b=b^{-1}a^2b=a^{-2}$, so $|a|=4$. Since $\langle a \rangle$ is an index-two subgroup of $G$, we must have $G \cong Q_8$. It is easy to check that $Q_8$ admits no ORR (in fact, Babai \cite[Theorem 2.1]{babai1} showed that it does not even admit a DRR).

Now suppose that $G$ is abelian, and $|a|>4$. Let $\Gamma=\Cay(G,S)$, with $S=\{a,a^2,b\}$; the condition on $|a|$ ensures that $\Gamma$ is an oriented Cayley digraph. Calculations, using the assumption that $d(G)=2$, show that $\Gamma[S]$ consists of a single arc from $a$ to $a^2$ with $b$ isolated, unless $a^2=b^2$. In particular, when $a^2\ne b^2$, $\Gamma[S]$  is asymmetric and hence, by Lemma~\ref{Watkins-Nowitz}, $\Gamma$ is an ORR for $G$. Suppose then $a^2=b^2$. Replacing $a$ by $a^{-1}$ in this argument produces the  ORR $\Cay(G,\{a^{-1},a^{-2},b\})$ for $G$ unless $a^{-2}=b^2$. We may then assume that $a^{-2}=b^2$. As $a^2=b^2$, we obtain $a^4=1$, contradicting $|a|>4$.

If $G$ is abelian, $|a|\in\{3,4\}$ and $|b|>4$, then reversing the roles of $a$ and $b$ in the previous paragraph produces an ORR for $G$. 

Finally, suppose that $G$ is abelian and $|a|, |b| \in \{3,4\}$. Since $d(G)=2$, $G$ is isomorphic to one of  $C_3 \times C_3$, $C_4 \times C_2$, or $C_4 \times C_4$. We can use \texttt{magma}~\cite{magma} to verify that the first and second group do not admit ORRs, and that the third group admits an ORR; for example, $\Cay(G, \{a,ab,a^2b,a^3b\})$ is an ORR.
\end{proof}

We can now complete the characterization of groups $G$ with $d(G)=2$.

\begin{theorem}\label{thrm:A1}
Let $G$ be a finite group with $d(G)=2$. Then one of the following holds:
\begin{enumerate}
\item[(i)] $G$ is generalized dihedral;
\item[(ii)] $G$ admits an ORR;
\item[(iii)] $G$ is isomorphic to one of the following groups:
\begin{itemize}
\item $C_3\times C_3$, 
\item $C_4\times C_2$, 
\item $Q_8$, or 
\item $\langle a,b\mid a^4=b^4=(ab)^2=(ab^{-1})^2=1\rangle$ (and $G$ has order $16$).
\end{itemize}
\end{enumerate}
Furthermore, the groups in (i) and (iii) admit no ORR.
\end{theorem}
\begin{proof}
Suppose that $G$ is not generalized dihedral and that $G$ admits no ORR. If $G$ is abelian or $G$ admits a generating set $\{a,b\}$ with $|a|, |b|, |ba^{-1}|>2$, then by Lemma~\ref{A:i-or-iv} $G$ is isomorphic to one of $C_3\times C_3$, $C_4\times C_2$, or $Q_8$, and hence $G$ satisfies~(iii). The only remaining possibility is that $G$ satisfies part~(iii) of Lemma~\ref{lemma:A2}. Now, Lemma~\ref{lemma:A3} shows that $G$ satisfies part~(iii) of this lemma.
\end{proof}

\section{Groups with $d(G)=3$}\label{3-gen}

We turn now to  groups $G$ with $d(G)=3$, since this is another case that we have to deal with individually for the proof of Theorem~\ref{thrmmain}. When we have a generating set  $\{a,b,c\}$ for $G$ with $|a|, |b|, |c|>2$, then adding some of the elements $ab^{-1}$, $bc^{-1}$, and $ac^{-1}$ to a connection set is often helpful in producing an ORR. Of course, the resulting Cayley digraph will not be an oriented Cayley digraph if any of the elements added is an involution. Our analysis of groups $G$ with $d(G)=3$ therefore relies heavily on whether or not the elements $ab^{-1}$, $bc^{-1}$, and $ac^{-1}$ are involutions. Since in $\{a,b,c\}$ we can replace any of $a$, $b$, and $c$ by their inverses and still have a generating set for $G$, we are also interested in whether or not $ab$, $bc$, and $ac$ are involutions. We begin with a classification of  groups $G$ with $d(G)=3$ that we will use in determining which of them admit ORRs.

\begin{lemma}\label{lemma:2A}Let $G$ be a  group with $d(G)=3$. Then one of the following holds:
\begin{enumerate}
\item[(i)]$G$ is generalized dihedral;
\item[(ii)]$G$ is abelian and admits a generating set $\{a,b,c\}$ with $|a|, |b|, |c|, |ba^{-1}|,$ $|cb^{-1}|>2$;
\item[(iii)] every generating set $\{a,b,c\}$ for $G$ with $|a|,|b|,|c|>2$ has $|ab|=|ab^{-1}|=|bc|=|bc^{-1}|=|ac|=|ac^{-1}|=2$;
\item[(iv)]$G$ admits a generating set $\{a,b,c\}$ with $|a|,|b|,|c|>2$, $|ab|=|ab^{-1}|=|bc|=|bc^{-1}|=2$, $|ac^{-1}|>2$ and $[a,c]\ne 1$;
\item[(v)]$G$ admits a generating set $\{a,b,c\}$ with $|a|,|b|,|c|>2$, $|ab|=|ab^{-1}|=|bc|=|bc^{-1}|=2$, $|ac^{-1}|>2$ and $[a,c]= 1$;
\item[(vi)]$G$ admits a generating set $\{a,b,c\}$ with $|a|,|b|,|c|,|ba^{-1}|,|cb^{-1}|>2$ and $[a,b]\ne 1$.
\end{enumerate}
\end{lemma}
\begin{proof}
Suppose that $G$ is not generalized dihedral. 
By Lemma~\ref{lemma:A1}, $G$ admits a generating set $\{a_1,a_2,a_3\}$ with $|a_1|,|a_2|,|a_3|>2$.

Suppose first that $G$ is abelian and that $|a_2a_1^{-1}|>2$. If $|a_3a_2^{-1}|>2$, then (ii) holds by using $a=a_1$, $b=a_2$ and $c=a_3$; similarly, if $|a_3^{-1}a_2^{-1}|>2$, then we see that (ii) holds by using $a=a_1$, $b=a_2$, $c=a_3^{-1}$. Otherwise, $|a_2a_3^{-1}|=|a_2a_3|=2$. Likewise, if $|a_3a_1^{-1}|>2$ or $|a_3^{-1}a_1^{-1}|>2$, then using $a=a_2$, $b=a_1$ and $c\in \{a_3,a_3^{-1}\}$, we have (ii). The only remaining possibility is $|a_2a_3^{-1}|=|a_2a_3|=|a_3a_1^{-1}|=|a_1a_3|=2$. In this case, since $G$ is abelian, (v) holds with $a=a_2$, $c=a_1$, and $b=a_3$. 

Now we may assume that, if $G$ is abelian, then $G$ does not admit any generating set $\{a_1,a_2,a_3\}$ with $|a_1|, |a_2|, |a_3|>2$ and $|a_2a_1^{-1}|>2$. Then if $\{a,b,c\}$ is a generating set for $G$ with $|a|, |b|, |c|>2$, taking $(a_1,a_2,a_3)=(b,a,c)$ gives $|ab^{-1}|=2$; $(a_1,a_2,a_3)=(b^{-1},a,c)$ gives $|ab|=2$; $(a_1,a_2,a_3)=(c,b,a)$ gives $|bc^{-1}|=2$; $(a_1,a_2,a_3)=(c^{-1},b,a)$ gives $|bc|=2$; $(a_1,a_2,a_3)=(c,a,b)$ gives $|ac^{-1}|=2$; and $(a_1,a_2,a_3)=(c^{-1},a,b)$ gives $|ac|=2$, so (iii) holds. This completes the proof if $G$ is abelian.

\smallskip

Suppose that for every $i\in \{1,2,3\}$, there exists $j_i\in \{1,2,3\}\setminus\{i\}$ and $\varepsilon_i\in \{1,-1\}$ such that $|a_ia_{j_i}^{-\varepsilon_i}|>2$. In particular, $i\mapsto j_i$ is a bijective function  on $\{1,2,3\}$ with no fixed points. Relabelling the indexed set $\{1,2,3\}$ if necessary we may assume that $j_1=3$, $j_2=1$ and $j_3=2$. Set $a_1'=a_1$, $a_2'=a_2^{\varepsilon_1}$ and $a_3'=a_3^{\varepsilon_1\varepsilon_2}$. By construction, $\{a_1',a_2',a_3'\}$ is a generating set for $G$ with $|a_1'|,|a_2'|,|a_3'|,|a_2'a_1'^{-1}|,|a_3'a_2'^{-1}|>2$, and either $|a_1'a_3'^{-1}|>2$ or $|a_1'a_3'|>2$. In particular, replacing the original generating set $\{a_1,a_2,a_3\}$ if necessary, we may assume that $|a_1|,|a_2|,|a_3|,|a_2a_1^{-1}|,|a_3a_2^{-1}|>2$, and either $|a_1a_3^{-1}|>2$ or $|a_1a_3|>2$.

If $[a_1,a_2]\ne 1$, then $\{a_1,a_2,a_3\}$ is a generating set for $G$ satisfying part~(vi). So we may assume that $[a_1,a_2]=1$. Now, $\{a_3,a_2,a_1\}$ is a generating set for $G$ with $|a_3|$, $|a_2|$, $|a_1|$, $|a_2a_3^{-1}|$, $|a_1a_2^{-1}|>2$. In particular, if $[a_3,a_2]\ne 1$, then $\{a_3,a_2,a_1\}$ is a generating set for $G$ satisfying part~(vi). So we may assume that $[a_3,a_2]=1$. As $G$ is not abelian, we may assume that $[a_1,a_3]\neq 1$.

Suppose $|a_1a_3^{-1}|>2$. Then $\{a_3,a_1,a_2\}$ is a generating set for $G$ and we have $|a_3|, |a_1|, |a_2|, |a_1a_3^{-1}|, |a_2a_1^{-1}|>2$. Thus $G$ together with the generating set $\{a_3,a_1,a_2\}$ satisfy~(vi). The only remaining possibility is that $|a_1a_3|>2$. Then $\{a_3^{-1},a_1,a_2\}$ is a generating set for $G$ with $|a_3^{-1}|, |a_1|, |a_2|, |a_1a_3|, |a_2a_1^{-1}|>2$. Thus $G$ together with the generating set $\{a_3^{-1},a_1,a_2\}$ satisfy part~(vi).

\smallskip

We may now assume that for every generating set $\{a_1,a_2,a_3\}$ for $G$ with $|a_1|, |a_2|,|a_3|>2$, there exists $i\in \{1,2,3\}$ with $|a_ia_j^{-1}|=|a_ia_j|=2$, for every $j\in \{1,2,3\}\setminus\{i\}$. Relabelling the index set $\{1,2,3\}$, we may assume that $i=2$. In particular, $|a_2a_1^{-1}|=|a_2a_1|=|a_2a_3^{-1}|=|a_2a_3|=2$. If $G$ has a generating set of this sort with $|a_1a_3^{-1}|>2$, then $G$ satisfies part~(iv) or~(v) (depending on whether $[a_1,a_3]\ne 1$ or $[a_1,a_3]=1$) by taking $(a,b,c)=(a_1,a_2,a_3)$  (observe that since $(ba)^2=1$, we have $1=a(ba)^2a^{-1}=(ab)^2aa^{-1}$, so $|ab|=2$). Similarly, if $G$ has a generating set of this sort with $|a_1a_3|>2$, then $G$ satisfies part~(iv) or~(v) (depending on whether $[a_1,a_3]\ne 1$ or $[a_1,a_3]=1$) by taking $(a,b,c)=(a_1,a_2,a_3^{-1})$. The only remaining possibility is that for every generating set of this sort for $G$, we have $|a_1a_3^{-1}|=|a_1a_3|=2$. Thus $G$ satisfies part~(iii) since $|a_1a_2|=|a_2a_1|=2$ and $|a_1a_2^{-1}|=|a_2a_1^{-1}|=2$. 
\end{proof}

We will now consider the various families listed in our classification, with the exception of the generalized dihedral groups, since we already know that these do not admit ORRs.
We begin with the family described in Lemma~\ref{lemma:2A}~(ii).

\begin{lemma}\label{abelian-3-gen}
Let $G$ be an  abelian group with $d(G)=3$ and a generating set $\{a,b,c\}$ with $|a|, |b|, |c|, |ba^{-1}|, |cb^{-1}|>2$. Then $G$ admits an ORR unless $G \cong C_3 \times C_2^3$.
\end{lemma}

\begin{proof}
We subdivide the proof in various cases.

\smallskip

\noindent\textsc{Case 1. }Suppose that $|ab|>2$. 

\smallskip 

\noindent In this case, let $S=\{a,b,c,ba^{-1},ab\}$. Since $|a|, |b|, |c|, |ba^{-1}|, |ab|>2$ (and so $ab \neq ba^{-1}, (ba^{-1})^{-1}$), we see that $S\cap S^{-1}=\emptyset$ and hence $\Cay(G,S)$ is an oriented Cayley digraph. 

Since $d(G)=3$, it is straightforward to observe that $c$ is an isolated vertex of $\Gamma[S]$ and  $(a,ab)$, $(b,ab)$, $(a,b)$, and $(ba^{-1},b)$ are arcs for $\Gamma[S]$: see Figure~\ref{TORONTO}. 
\begin{figure}[!hhh]
\begin{center}
\begin{tikzpicture}[node distance=1.5cm]
\node[circle,draw,inner sep=1pt, label=90:$ab$](A0){};
\node[right=of A0,circle,draw,inner sep=1pt, label=90:$a$](A1){};
\node[right=of A1,circle,draw,inner sep=1pt,label=90:$ba^{-1}$](A2){};
\node[below=of A1,circle,draw,inner sep=1pt,label=-90:$b$](A3){};
\node[below=of A2,circle,draw,inner sep=1pt,label=-90:$c$](A4){};
\draw[-angle 45, ] (A1) to  (A3);
\draw[-angle 45, ] (A3) to  (A0);
\draw[-angle 45, ] (A1) to  (A0);
\draw[-angle 45, ] (A2) to  (A3);
\end{tikzpicture}
\end{center}
\caption{Figure for the proof of Lemma~\ref{abelian-3-gen}}\label{TORONTO}
\end{figure}
Using the fact that $d(G)=3$, that $G$ is abelian and $|ba^{-1}|>2$, it is a routine computation to show that the only possible arcs of $\Gamma[S]$ (apart from the arcs drawn in Figure~\ref{TORONTO}) are $(ba^{-1},a)$, $(a,ba^{-1})$ and $(ab,ba^{-1})$. This requires only a tedious case-by-case analysis, we leave the details to the reader. Here we deal with only one particular case that in our opinion is rather representative. If $(ba^{-1},a)$ is an arc of $\Gamma[S]$, then $a(ba^{-1})=a^2b^{-1}\in S$; now using the irredundancy of $a,b,c$, we deduce that either $a^{2}b^{-1}=b$ or $a^2b^{-1}=ba^{-1}$. In the first case, $a^2=b^2$, but as $|ba^{-1}|>2$, we get $b^2a^{-2}\ne 1$, a contradiction. Therefore $a^3=b^2$ if $(ba^{-1},a)$ is an arc of $\Gamma[S]$. All other cases are similar.

Moreover, using the irredundancy of $a,b,c$, one of the following holds:
\begin{itemize}
\item $(ba^{-1},a)$ is an arc of $\Gamma[S]$ and $a^3=b^2$,
\item $(a,ba^{-1})$ is an arc of $\Gamma[S]$ and $a^3=1$,
\item $(ab,ba^{-1})$ is an arc of $\Gamma[S]$ and $a^3=1$. 
\end{itemize}
In particular, $\Gamma[S]$ is isomorphic to one of the graphs in Figure~\ref{TORONTO1}.
\begin{figure}[!hhh]
\begin{center}
\begin{tikzpicture}[node distance=.9cm]
\node[circle,draw,inner sep=1pt, label=90:$ab$](A0){};
\node[right=of A0,circle,draw,inner sep=1pt, label=90:$a$](A1){};
\node[right=of A1,circle,draw,inner sep=1pt,label=90:$ba^{-1}$](A2){};
\node[below=of A1,circle,draw,inner sep=1pt,label=-90:$b$](A3){};
\node[below=of A2,circle,draw,inner sep=1pt,label=-90:$c$](A4){};
\draw[-angle 45, ] (A1) to  (A3);
\draw[-angle 45, ] (A3) to  (A0);
\draw[-angle 45, ] (A1) to  (A0);
\draw[-angle 45, ] (A2) to  (A3);
\node[right=of A2,circle,draw,inner sep=1pt, label=90:$ab$](AA0){};
\node[right=of AA0,circle,draw,inner sep=1pt, label=90:$a$](AA1){};
\node[right=of AA1,circle,draw,inner sep=1pt,label=90:$ba^{-1}$](AA2){};
\node[below=of AA1,circle,draw,inner sep=1pt,label=-90:$b$](AA3){};
\node[below=of AA2,circle,draw,inner sep=1pt,label=-90:$c$](AA4){};
\draw[-angle 45, ] (AA1) to  (AA3);
\draw[-angle 45, ] (AA3) to  (AA0);
\draw[-angle 45, ] (AA1) to  (AA0);
\draw[-angle 45, ] (AA2) to  (AA3);
\draw[-angle 45, ] (AA2) to (AA1);
\node[right=of AA2,circle,draw,inner sep=1pt, label=90:$ab$](AAA0){};
\node[right=of AAA0,circle,draw,inner sep=1pt, label=90:$a$](AAA1){};
\node[right=of AAA1,circle,draw,inner sep=1pt,label=90:$ba^{-1}$](AAA2){};
\node[below=of AAA1,circle,draw,inner sep=1pt,label=-90:$b$](AAA3){};
\node[below=of AAA2,circle,draw,inner sep=1pt,label=-90:$c$](AAA4){};
\draw[-angle 45, ] (AAA1) to  (AAA3);
\draw[-angle 45, ] (AAA3) to  (AAA0);
\draw[-angle 45, ] (AAA1) to  (AAA0);
\draw[-angle 45, ] (AAA2) to  (AAA3);
\draw[-angle 45, ] (AAA1) to (AAA2);
\path[-angle 45, ] (AAA0) edge [bend left=99] (AAA2); 
\end{tikzpicture}
\end{center}
\caption{Figure for the proof of Lemma~\ref{abelian-3-gen}}\label{TORONTO1}
\end{figure}
If $\Gamma[S]$ is one of the first two graphs in Figure~\ref{TORONTO1}, then $\Gamma[S]$ is asymmetric, so by Lemma~\ref{Watkins-Nowitz}, $\Gamma$ is an ORR for $G$. Therefore we may assume that $\Gamma[S]$ is the third graph in Figure~\ref{TORONTO1}, and in particular $a^3=1$.

Now, consider $S'=\{a,b,c,ba^{-1}, cb^{-1}, ab\}$ and $\Gamma'=\Cay(G, S')$. Clearly, like $\Gamma[S]$, the graph $\Gamma'[S']$ also has arcs $(a,b)$, $(ba^{-1},b)$, $(a,ab)$, $(b,ab)$, $(ab,ba^{-1})$, and $(a,ba^{-1})$, but it also has arcs $(b,c)$ and $(cb^{-1},c)$. See Figure~\ref{TORONTO2}.
\begin{figure}[!hhh]
\begin{center}
\begin{tikzpicture}[node distance=1.5cm]
\node[circle,draw,inner sep=1pt, label=90:$ab$](A0){};
\node[right=of A0,circle,draw,inner sep=1pt, label=90:$a$](A1){};
\node[right=of A1,circle,draw,inner sep=1pt,label=90:$ba^{-1}$](A2){};
\node[below=of A1,circle,draw,inner sep=1pt,label=-90:$b$](A3){};
\node[below=of A2,circle,draw,inner sep=1pt,label=-90:$c$](A4){};
\node[right=of A4,circle,draw,inner sep=1pt,label=-90:$cb^{-1}$](A5){};
\draw[-angle 45, ] (A1) to  (A3);
\draw[-angle 45, ] (A3) to  (A0);
\draw[-angle 45, ] (A1) to  (A0);
\draw[-angle 45, ] (A2) to  (A3);
\path[-angle 45, ] (A0) edge [bend left=99] (A2); 
\draw[-angle 45, ](A1) to (A2);
\draw[-angle 45, ](A3) to (A4);
\draw[-angle 45, ](A5) to (A4);
\end{tikzpicture}
\end{center}
\caption{Figure for the proof of Lemma~\ref{abelian-3-gen}}\label{TORONTO2}
\end{figure}
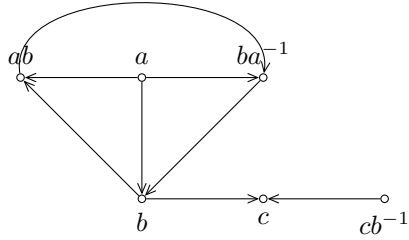

Arguing as above, using $d(G)=3$,  the fact that $G$ is abelian, $|a|=3$ and $|cb^{-1}|>2$, it is straightforward to verify that except for the arcs listed above $(cb^{-1},b)$ is the only other possible arc of $\Gamma[S']$ (arising only if the relation $b^3=c^2$ is satisfied in $G$). In any case,  $\Gamma'[S']$ is asymmetric, and by Lemma~\ref{Watkins-Nowitz}, $\Gamma'$ is an ORR for $G$.

\smallskip

\noindent\textsc{Case 2. }Suppose that $|cb|>2$. 

\smallskip 

\noindent We can repeat the argument in Case~$1$ reversing the roles of $c$ and $a$ since $|bc^{-1}|=|cb^{-1}|$ and $|ab^{-1}|=|ba^{-1}|$.

\smallskip

From Cases~$1$ and~$2$, we may assume that $|ab|=|cb|=2$.  Thus $a^2b^2=1=c^2b^2$ and hence $a^2=b^{-2}=c^2$. The fact that $d(G)=3$ implies that $|a|$ is even. We have $(ba^{-1})^2=b^2a^{-2}=a^{-4}$ so, since $|ba^{-1}|>2$, we must have $|a|>4$.

\smallskip

\noindent\textsc{Case 3. }$|a|>8$.

\smallskip 

\noindent   Let $S=\{a,b,c,ba^{-1},a^{-4}\}$ and $\Gamma=\Cay(G,S)$. Since $a^{-4}=(ba^{-1})^2 \neq a^2=c^2$, using $d(G)=3$ we see that $c$ is the unique isolated vertex of $\Gamma[S]$, and is fixed by any automorphism of $\Gamma[S]$. Calculations show that $(ba^{-1},a^{-4})$ is the only arc involving $a^{-4}$, whereas $b$ has two in-neighbours and $ba^{-1}$ and $a$ each have at least one out-neighbour, so $a^{-4}$ is also fixed by any automorphism of $\Gamma[S]$, as is its unique in-neighbour $ba^{-1}$. Since there is an arc $(a,b)$ there cannot be an automorphism of $\Gamma[S]$ exchanging these vertices, so $\Gamma[S]$ is asymmetric and by Lemma~\ref{Watkins-Nowitz}, $G$ admits an ORR. 

\smallskip

\noindent\textsc{Case 4. }$|a|\in \{6,8\}$.

\smallskip 

\noindent If $|a|=6$, then $G=\langle a,b,c\rangle=\langle a, ab,cb\rangle =\langle a\rangle\times \langle ab\rangle\times \langle cb\rangle\cong C_6\times C_2^2 \cong C_3 \times C_2^3$, the exceptional case. Indeed, a computation with \texttt{magma} shows that $C_3\times C_2^3$ has no ORR.

If $|a|=8$, then $G=\langle a,b,c\rangle=\langle a,ab,cb\rangle =\langle a\rangle\times \langle ab\rangle\times \langle cb\rangle\cong C_8 \times C_2^2$ and a computation with \texttt{magma}~\cite{magma} shows that $G$ admits an ORR. For instance, $\Cay(G,\{a,a^3,a^3b,abc\})$ is an ORR.
\end{proof}

Next we consider the family of groups described in Lemma~\ref{lemma:2A}(iii).

\begin{lemma}\label{lemma:6}Let $G$ be a  group, not generalized dihedral, with $d(G)=3$ and such that, for every generating set $\{a,b,c\}$ with $|a|,|b|,|c|>2$, we have
$$|ab|=|ab^{-1}|=|bc|=|bc^{-1}|=|ac|=|ac^{-1}|=2.$$
Then $G$ is a quotient of the group of order $64$ with presentation 
\begin{align*}
\langle x,y,z\mid &x^4=y^4=z^4=(xy)^2=1,\\
&(xy^{-1})^2=(xz)^2=(xz^{-1})^2=(yz)^2=(yz^{-1})^2=1\rangle.
\end{align*} 
Moreover, one of the following holds:
\begin{enumerate}
\item[(i)]$G$ admits an ORR,
\item[(ii)]$G=C_4\times C_2^2$ and $G$ admits no ORR, or
\item[(iii)]$G$ has order $32$, presentation 
\begin{align*}
\langle x,y,z\mid &x^4=y^4=z^4=(xy)^2=(xy^{-1})^2=1,\\
&(xz)^2=(xz^{-1})^2=(yz)^2=(yz^{-1})^2=x^2y^2z^2=1\rangle
\end{align*} and admits no ORR.
\end{enumerate}
\end{lemma}
\begin{proof}
Since $G$ is not generalized dihedral and $d(G)=3$, there exists a generating set $\{a,b,c\}$ of $G$ with $|a|,|b|,|c|>2$ by Lemma~\ref{lemma:A1}.

Consider the generating set $\{a,ba^2,c\}$ of $G$. From Lemma~\ref{lemma:A-1} applied to $a$ and $b$, we have 
\begin{equation*}
(ba^2)^2=ba^2ba^2=b^2(b^{-1}a^2ba^2)=b^2(a^{-2}a^2)=b^2.
\end{equation*}
Since $b^2\ne 1$, we deduce $|ba^2|>2$. As $|a|,|ba^2|,|c|>2$, we are in the position to apply the  hypothesis of this lemma to the generating set $\{a,ba^2,c\}$, so that $(ba^2)c^{-1}$ is an involution. Lemma~\ref{lemma:A-1} applied to $a$ and $c$ yields $a^2c^{-1}=c^{-1}a^{-2}$; hence we deduce
$$1=((ba^2)c^{-1})^2=ba^2c^{-1}ba^2c^{-1}=bc^{-1}a^{-2}bc^{-1}a^{-2}$$
and $$(bc^{-1})a^{-2}(bc^{-1})=a^2.$$ Therefore $bc^{-1}$ acts by conjugation inverting $a^{-2}$. Lemma~\ref{lemma:A-1} applied to $a,b$ and to $a,c$ yields that both $b$ and $c$ act by conjugation inverting $a^2$. Thence $$a^{2}=(a^{-2})^{bc^{-1}}=((a^{-2})^b)^{c^{-1}}=(a^{2})^{c^{-1}}=a^{-2}.$$ This is possible if and only if  $|a|=4$. An entirely symmetric argument (applied to the generating sets $\{ab^2,b,c\}$ and $\{ac^2,b,c\}$) shows that $|b|=|c|=4$.

This proves that $G$ is a quotient of the group $P$ with presentation 
\begin{align*}
P=\langle x,y,z\mid &x^4=y^4=z^4=1,\\
&(xy)^2=(xy^{-1})^2=(xz)^2=(xz^{-1})^2=(yz)^2=(yz^{-1})^2=1\rangle.
\end{align*}
A computation with \texttt{magma}~\cite{magma} shows that $P$ has order $64$. Moreover, another computation with \texttt{magma}~\cite{magma} shows that the only quotients  $G$ of $P$ with $d(G)=3$ and with $G$ generated by three non-involutions that do not admit ORRs are the two groups listed in~(ii) and~(iii).
\end{proof}

The following lemma will prove useful in a few situations that follow.

\begin{lemma}\label{Q8-inside}
Let $G$ be a  group with $d(G)=3$ and with a generating set $\{a,b,c\}$ such that $|a|,|b|,|c|,|ba^{-1}|>2$ and $[a,b]\neq 1$.  Then either $\Cay(G, \{a,b,c,ba^{-1}\})$ is an ORR for $G$, or $\langle a, b \rangle \cong Q_8$.
\end{lemma}

\begin{proof}
Consider the set $S=\{a,b,c,ba^{-1}\}$. By hypothesis $S\cap S^{-1}=\emptyset$ and hence $\Gamma=\Cay(G,S)$ is an oriented Cayley digraph. 

Consider the graph $\Gamma[S]$. As $ba^{-1}\in S$, we see that $(a,b)$ is an arc of $\Gamma[S]$. Using the fact that $d(G)=3$, calculations show that $c$ is an isolated vertex in $\Gamma[S]$. 
The only other possible isolated vertex of $\Gamma[S]$ is $ba^{-1}$. 

If $ba^{-1}$ is isolated in $\Gamma[S]$, then $\Gamma[S]$ is the digraph on four vertices with a single arc from $a$ to $b$. It is then clear that $\Aut(\Gamma)_1$ must fix $a$ and $b$, so by Lemma~\ref{Watkins-Nowitz} it also fixes $\langle a,b\rangle$  pointwise and hence $ba^{-1}$. Therefore $\Aut(\Gamma)_1$ also fixes $c$ as the only other vertex of $S$. Applying Lemma~\ref{Watkins-Nowitz} again, we see that $\Gamma$ is an ORR for $G$. 

If on the other hand $ba^{-1}$ is not isolated, then $\Aut(\Gamma)_1$ must fix $c$, so fixes $\{a,b,ba^{-1}\}$ setwise. Now, by Lemma~\ref{A:i-or-iv}~(i) applied to the group $\langle a,b\rangle$, we obtain that  either $\Gamma[\{a,b,ba^{-1}\}]$ is asymmetric, so that by Lemma~\ref{Watkins-Nowitz} $\Gamma$ is an ORR for $G$, or $\langle a, b \rangle \cong Q_8$.
\end{proof}

We can now look at the groups described in Lemma~\ref{lemma:2A}~(iv).

\begin{lemma}\label{lemma:B1}
Let $G$ be a  group with $d(G)=3$ and with a generating set $\{a,b,c\}$ such that $$|a|,|b|,|c|,|ca^{-1}|>2,\, |ba|=|ba^{-1}|=|bc|=|bc^{-1}|=2\quad\textrm{ and}\quad [a,c]\ne 1.$$ Then one of the following holds
\begin{enumerate}
\item[(i)]$\Cay(G,\{a,b,c,ca^{-1}\})$ is an ORR;
\item[(ii)]$\langle a,c\rangle\cong Q_8$, $|b|>4$ and $\Cay(G,\{a,b,b^2,c,ca^{-1},b^2a\})$ is an ORR;
\item[(iii)]$G$ is a quotient of the group of order $32$ with presentation 
\begin{align*}
\langle x,y,z\mid &x^4=y^4=z^4=(yx)^2=(yx^{-1})^2=(yz)^2=(yz^{-1})^2=1,\\
&x^2=z^2,x^z=x^{-1} \rangle.
\end{align*} Moreover either $G$ admits an ORR or $G$ has presentation
\begin{align*}\langle x,y,z\mid &x^4=y^4=z^4=(yx)^2=(yx^{-1})^2=(yz)^2=(yz^{-1})^2=1,\\
&x^2=z^2,x^z=x^{-1}, x^2=y^2\rangle
\end{align*}
and has order $16$.
\end{enumerate}
\end{lemma}

\begin{proof}
Observe that Lemma~\ref{lemma:A-1} applied first to  $a,b$ and then to $b,c$ yields $|a|,|b|,|c|\geq 4$.
By Lemma~\ref{Q8-inside}, either (i) occurs, or $\langle a, c \rangle \cong Q_8$.
Therefore, in the rest of the proof, we may assume that $\langle a,c\rangle\cong Q_8$.

Now assume that $|b|>4$. Let $S=\{a,c,ca^{-1},b,b^2,b^2a\}$ and $\Gamma=\Cay(G,S)$. From Lemma \ref{lemma:A-1} applied to $a,b$ we get $(b^2)^a=b^{-2}$, and hence $(b^2a)^2=b^2ab^2a=a^2\neq 1$. Thus $|b^2a|>2$.
As $|a|,|c|,|ca^{-1}|,|b|,|b^2|,|b^2a|>2$, we have $S\cap S^{-1}=\emptyset$ and $\Gamma$ is an oriented Cayley digraph. Observe that $(b,b^2)$, $(a,b^2a)$ and $(a,c)$ are  arcs of $\Gamma[S]$. As $\langle a,c\rangle\cong Q_8$, it is easy to verify that $(c,ca^{-1})$ and $(ca^{-1},a)$ are also arcs of $\Gamma[S]$. Moreover, $ab^{-2}=b^2a\in S$ and hence $(b^2,a)$ is an arc of $\Gamma[S]$.  Thus
$$(a,c),(c,ca^{-1}),(ca^{-1},a),(b,b^2),(b^2,a),(a,b^2a)\quad\textrm{are arcs of }\Gamma[S].$$
Using the fact that $a,b,c$ is irredundant, $\langle a,c\rangle\cong Q_8$ and $|b|>4$,  it is a routine computation to show that the only possible arcs of $\Gamma[S]$ (apart from the arcs listed above) are $(b^2a,b)$ and $(b^2a,b^2)$. (See Figure~\ref{Fig3}.)
\begin{figure}[!hhh]
\begin{center}
\begin{tikzpicture}[node distance=1.5cm]
\node[circle,draw,inner sep=1pt, label=-90:$a$](A0){};
\node[right=of A0,circle,draw,inner sep=1pt, label=-90:$ca^{-1}$](A1){};
\node[above=of A0,circle,draw,inner sep=1pt,label=90:$c$](A2){};
\node[left=of A0,circle,draw,inner sep=1pt,label=-90:$b^2$](A3){};
\node[above=of A3,circle,draw,inner sep=1pt,label=90:$b^2a$](A4){};
\node[left=of A3,circle,draw,inner sep=1pt,label=-90:$b$](A5){};
\draw[-angle 45, ] (A0) to  (A2);
\draw[-angle 45, ] (A0) to  (A4);
\draw[-angle 45, ] (A1) to  (A0);
\draw[-angle 45, ] (A2) to  (A1);
\draw[-angle 45, ] (A3) to  (A0);
\draw[-angle 45, ] (A5) to  (A3);
\draw[-angle 45, ,dashed](A4) to (A5);
\draw[-angle 45, ,dashed](A4) to (A3);
\end{tikzpicture}
\end{center}
\caption{Figure for the proof of Lemma~\ref{lemma:B1}}\label{Fig3}
\end{figure}
  This requires only routine computations and a tedious case-by-case analysis, we leave the details to the reader. (Here we deal with only one particular case that in our opinion is rather representative. If $(b^2a,ca^{-1})$ is an arc of $\Gamma[S]$, then $ca^{-2}b^{-2}\in S$; now using first $\langle a,c\rangle\cong Q_8$ and then the irredundancy of $a,b,c$, we deduce that $ca^{-2}b^{-2}=c^{-1}b^{-2}$ must be the element $c$ of $S$, and hence $b^2=c^{-2}$. Thus $b^4=c^{-4}=1$ and $|b|\le |c|=4$, contradicting $|b|>4$. All other cases are similar).  Therefore $\Gamma[S]$ is one of the four graphs shown in Figure~\ref{Fig4}.
\begin{figure}[!hhh]
\begin{center}
\begin{tikzpicture}[node distance=0.725cm]
\node[circle,draw,inner sep=1pt, label=-90:$a$](A0){};
\node[right=of A0,circle,draw,inner sep=1pt, label=-90:$ca^{-1}$](A1){};
\node[above=of A0,circle,draw,inner sep=1pt,label=90:$c$](A2){};
\node[left=of A0,circle,draw,inner sep=1pt,label=-90:$b^2$](A3){};
\node[above=of A3,circle,draw,inner sep=1pt,label=90:$b^2a$](A4){};
\node[left=of A3,circle,draw,inner sep=1pt,label=-90:$b$](A5){};
\draw[-angle 45, ] (A0) to  (A2);
\draw[-angle 45, ] (A0) to  (A4);
\draw[-angle 45, ] (A1) to  (A0);
\draw[-angle 45, ] (A2) to  (A1);
\draw[-angle 45, ] (A3) to  (A0);
\draw[-angle 45, ] (A5) to  (A3);
\node[left=of A0,circle](A00){};
\node[left=of A00,circle](A000){};
\node[left=of A000,circle,draw,inner sep=1pt, label=-90:$a$](AA0){};
\node[right=of AA0,circle,draw,inner sep=1pt, label=-90:$ca^{-1}$](AA1){};
\node[above=of AA0,circle,draw,inner sep=1pt,label=90:$c$](AA2){};
\node[left=of AA0,circle,draw,inner sep=1pt,label=-90:$b^2$](AA3){};
\node[above=of AA3,circle,draw,inner sep=1pt,label=90:$b^2a$](AA4){};
\node[left=of AA3,circle,draw,inner sep=1pt,label=-90:$b$](AA5){};
\draw[-angle 45, ] (AA0) to  (AA2);
\draw[-angle 45, ] (AA0) to  (AA4);
\draw[-angle 45, ] (AA1) to  (AA0);
\draw[-angle 45, ] (AA2) to  (AA1);
\draw[-angle 45, ] (AA3) to  (AA0);
\draw[-angle 45, ] (AA5) to  (AA3);
\draw[-angle 45, ](AA4) to (AA5);
\node[left=of AA0,circle](AA00){};
\node[left=of AA00,circle](AA000){};
\node[left=of AA000,circle,draw,inner sep=1pt, label=-90:$a$](AAA0){};
\node[right=of AAA0,circle,draw,inner sep=1pt, label=-90:$ca^{-1}$](AAA1){};
\node[above=of AAA0,circle,draw,inner sep=1pt,label=90:$c$](AAA2){};
\node[left=of AAA0,circle,draw,inner sep=1pt,label=-90:$b^2$](AAA3){};
\node[above=of AAA3,circle,draw,inner sep=1pt,label=90:$b^2a$](AAA4){};
\node[left=of AAA3,circle,draw,inner sep=1pt,label=-90:$b$](AAA5){};
\draw[-angle 45, ] (AAA0) to  (AAA2);
\draw[-angle 45, ] (AAA0) to  (AAA4);
\draw[-angle 45, ] (AAA1) to  (AAA0);
\draw[-angle 45, ] (AAA2) to  (AAA1);
\draw[-angle 45, ] (AAA3) to  (AAA0);
\draw[-angle 45, ] (AAA5) to  (AAA3);
\draw[-angle 45, ](AAA4) to (AAA3);
\node[left=of AAA0,circle](AAA00){};
\node[left=of AAA00,circle](AAA000){};
\node[left=of AAA000,circle,draw,inner sep=1pt, label=-90:$a$](AAAA0){};
\node[right=of AAAA0,circle,draw,inner sep=1pt, label=-90:$ca^{-1}$](AAAA1){};
\node[above=of AAAA0,circle,draw,inner sep=1pt,label=90:$c$](AAAA2){};
\node[left=of AAAA0,circle,draw,inner sep=1pt,label=-90:$b^2$](AAAA3){};
\node[above=of AAAA3,circle,draw,inner sep=1pt,label=90:$b^2a$](AAAA4){};
\node[left=of AAAA3,circle,draw,inner sep=1pt,label=-90:$b$](AAAA5){};
\draw[-angle 45, ] (AAAA0) to  (AAAA2);
\draw[-angle 45, ] (AAAA0) to  (AAAA4);
\draw[-angle 45, ] (AAAA1) to  (AAAA0);
\draw[-angle 45, ] (AAAA2) to  (AAAA1);
\draw[-angle 45, ] (AAAA3) to  (AAAA0);
\draw[-angle 45, ] (AAAA5) to  (AAAA3);
\draw[-angle 45, ](AAAA4) to (AAAA5);
\draw[-angle 45, ](AAAA4) to (AAAA3);
\end{tikzpicture}
\end{center}
\caption{Figure for the proof of Lemma~\ref{lemma:B1}}\label{Fig4}
\end{figure}
In all cases, $\Gamma[S]$ is asymmetric and hence by Lemma~\ref{Watkins-Nowitz}, $\Gamma$ is an ORR.

Summing up, if $\langle a,c\rangle\cong Q_8$ and $|b|>4$, then $\Cay(G,\{a,c,ca^{-1},b,b^2,b^2a\})$ is an ORR and part~(ii) holds. Therefore, in the rest of the proof, we may assume that $\langle a,c\rangle\cong Q_8$ and $|b|=4$. (Recall that $|b|\geq 4$.)

Now $G$ is a quotient of the group $P$ with presentation
\begin{align*}
P=\langle x,y,z\mid &x^4=y^4=z^4=(yx)^2=(yx^{-1})^2=(yz)^2=(yz^{-1})^2=1,\\
&x^2=z^2,x^z=x^{-1} \rangle.
\end{align*}
A computation with \texttt{magma} \cite{magma} shows that $P$ has order $32$. Moreover, another computation with \texttt{magma} \cite{magma} shows that the only quotient $G$ of $P$ with $d(G)=3$ and with $G$ generated by three non-involutions that does not admit an ORR is the group listed in~(iii).
\end{proof}

Our next lemma deals with almost the same situation as Lemma~\ref{lemma:B1}, but where $a$ and $c$ do commute, as described in Lemma~\ref{lemma:2A}~(v).

\begin{lemma}\label{lemma:BB1}
Let $G$ be a  group with generating set $\{a,b,c\}$ such that $$|a|,|b|,|c|,|ca^{-1}|>2, \,|ba|=|ba^{-1}|=|bc|=|bc^{-1}|=2\quad \textrm{and} \quad[a,c]= 1.$$ Then $G$ admits an ORR. In particular, one of the following holds
\begin{enumerate}
\item[(i)]$\Cay(G,\{a,c,cb^{-2},b,b^2\})$ is an ORR,
\item[(ii)]$|b|=4$, $|a|>4$, $|ca^{-2}|>2$ and $\Cay(G,\{a,a^2,c,ca^{-2},b\})$ is an ORR,
\item[(iii)]$|b|=4$, $|c|>4$, $|ac^{-2}|>2$ and $\Cay(G,\{a,c,c^2,ac^{-2},b\})$ is an ORR, or
\item[(iv)]$|G|\le 64$ and $G$ admits an ORR.
\end{enumerate}
\end{lemma}
\begin{proof} Observe that by Lemma~\ref{lemma:A-1} applied to $a$ and $b$, and to $b$ and $c$, since $|a|, |b|, |c|>2$ we must have that $|a|, |b|, |c|$ are all even and at least $4$.
We subdivide the proof into various cases.
\smallskip

\noindent\textsc{Case 1: }$|b|> 4$.

\smallskip

\noindent Consider the set $S=\{a,c,cb^{-2},b,b^2\}$ and $\Gamma=\Cay(G,S)$. Lemma \ref{lemma:A-1} applied to $b,c$ yields $(b^2)^c=b^{-2}$,  hence
$$(cb^{-2})^2=cb^{-2}cb^{-2}=c^2\neq 1$$
and $|cb^{-2}|>2$.  Thus $S\cap S^{-1}=\emptyset$ and hence $\Gamma=\Cay(G,S)$ is an oriented Cayley digraph. 

Consider the subgraph $\Gamma[S]$. It is easy to verify that $(b,b^2,c,cb^{-2})$ is a directed path of length $3$ in $\Gamma[S]$. Using the fact that $d(G)=3$, $|ab^{-1}|=|ab|=|bc|=|bc^{-1}|=2$, $[a,c]=1$ and $|b|>4$,  it is a routine computation to show that the only possible arcs of $\Gamma[S]$ (apart from the arcs in the path $(b,b^2,c,cb^{-2})$) are $(cb^{-2},b)$ and $(cb^{-2},b^2)$. (See Figure~\ref{Fig5}.)
\begin{figure}[!hhh]
\begin{center}
\begin{tikzpicture}[node distance=1.5cm]
\node[circle,draw,inner sep=1pt, label=90:$a$](A0){};
\node[right=of A0,circle,draw,inner sep=1pt, label=90:$b$](A1){};
\node[right=of A1,circle,draw,inner sep=1pt,label=90:$b^2$](A2){};
\node[below=of A1,circle,draw,inner sep=1pt,label=-90:$cb^{-2}$](A3){};
\node[below=of A2,circle,draw,inner sep=1pt,label=-90:$c$](A4){};
\draw[-angle 45, ] (A1) to  (A2);
\draw[-angle 45, ] (A2) to  (A4);
\draw[-angle 45, ] (A4) to  (A3);
\draw[-angle 45, ,dashed] (A3) to  (A1);
\draw[-angle 45, ,dashed] (A3) to  (A2);
\end{tikzpicture}
\end{center}
\caption{Figure for the proof of Lemma~\ref{lemma:BB1}, Case 1}\label{Fig5}
\end{figure}
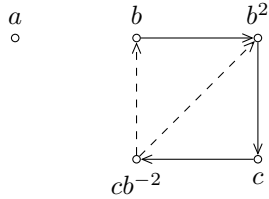
  This requires only routine computations and a tedious case-by-case analysis, we leave the details to the reader. Here we deal with only one particular case that in our opinion is the hardest. If $(a,b^2)$ is an arc of $\Gamma[S]$, then $b^2a^{-1}\in S$; now using the irredundancy of $a,b,c$, we deduce $b^2a^{-1}=a$ and hence $b^2=a^{2}$. From Lemma~\ref{lemma:A-1} applied to $a,b$, we deduce $(b^2)^{a}=b^{-2}$. Therefore $$b^{-2}=(b^2)^a=(a^2)^a=a^2=b^2$$
and hence $b^4=1$, contradicting $|b|>4$. All other cases are similar.  Therefore $\Gamma[S]$ is one of the four graphs shown in Figure~\ref{Fig6}.

\begin{figure}[h]
\begin{center}
\begin{tikzpicture}[node distance=0.8cm]
\node[circle,draw,inner sep=1pt, label=90:$a$](A0){};
\node[right=of A0,circle,draw,inner sep=1pt, label=90:$b$](A1){};
\node[right=of A1,circle,draw,inner sep=1pt,label=90:$b^2$](A2){};
\node[below=of A1,circle,draw,inner sep=1pt,label=-90:$cb^{-2}$](A3){};
\node[below=of A2,circle,draw,inner sep=1pt,label=-90:$c$](A4){};
\draw[-angle 45, ] (A1) to  (A2);
\draw[-angle 45, ] (A2) to  (A4);
\draw[-angle 45, ] (A4) to  (A3);
\node[right=of A2,circle,draw,inner sep=1pt, label=90:$a$](AA0){};
\node[right=of AA0,circle,draw,inner sep=1pt, label=90:$b$](AA1){};
\node[right=of AA1,circle,draw,inner sep=1pt,label=90:$b^2$](AA2){};
\node[below=of AA1,circle,draw,inner sep=1pt,label=-90:$cb^{-2}$](AA3){};
\node[below=of AA2,circle,draw,inner sep=1pt,label=-90:$c$](AA4){};
\draw[-angle 45, ] (AA1) to  (AA2);
\draw[-angle 45, ] (AA2) to  (AA4);
\draw[-angle 45, ] (AA4) to  (AA3);
\draw[-angle 45, ] (AA3) to  (AA1);
\node[right=of AA2,circle,draw,inner sep=1pt, label=90:$a$](AAA0){};
\node[right=of AAA0,circle,draw,inner sep=1pt, label=90:$b$](AAA1){};
\node[right=of AAA1,circle,draw,inner sep=1pt,label=90:$b^2$](AAA2){};
\node[below=of AAA1,circle,draw,inner sep=1pt,label=-90:$cb^{-2}$](AAA3){};
\node[below=of AAA2,circle,draw,inner sep=1pt,label=-90:$c$](AAA4){};
\draw[-angle 45, ] (AAA1) to  (AAA2);
\draw[-angle 45, ] (AAA2) to  (AAA4);
\draw[-angle 45, ] (AAA4) to  (AAA3);
\draw[-angle 45, ] (AAA3) to  (AAA2);
\node[right=of AAA2,circle,draw,inner sep=1pt, label=90:$a$](AAAA0){};
\node[right=of AAAA0,circle,draw,inner sep=1pt, label=90:$b$](AAAA1){};
\node[right=of AAAA1,circle,draw,inner sep=1pt,label=90:$b^2$](AAAA2){};
\node[below=of AAAA1,circle,draw,inner sep=1pt,label=-90:$cb^{-2}$](AAAA3){};
\node[below=of AAAA2,circle,draw,inner sep=1pt,label=-90:$c$](AAAA4){};
\draw[-angle 45, ] (AAAA1) to  (AAAA2);
\draw[-angle 45, ] (AAAA2) to  (AAAA4);
\draw[-angle 45, ] (AAAA4) to  (AAAA3);
\draw[-angle 45, ] (AAAA3) to  (AAAA1);
\draw[-angle 45, ] (AAAA3) to  (AAAA2);
\end{tikzpicture}
\end{center}
\caption{Figure for the proof of Lemma~\ref{lemma:BB1}, Case 1}\label{Fig6}
\end{figure}
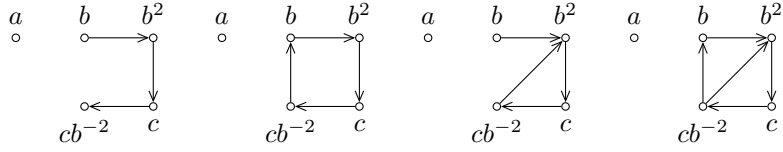

Except for the second graph in Figure~\ref{Fig6}, we see that $\Gamma[S]$ is asymmetric and hence by Lemma~\ref{Watkins-Nowitz} $\Gamma$ is an ORR. In particular, we may assume that $\Gamma[S]$ is the second graph in  Figure~\ref{Fig6}. This implies that $(cb^{-2},b)$ is an arc of $\Gamma[S]$; that is, $b(cb^{-2})^{-1}\in S$. Using $d(G)=3$ and Lemma~\ref{lemma:A-1} applied to $b,c$ so that $cb^{-2}=b^2c$, we obtain $b(cb^{-2})^{-1}=b^3c^{-1}=c\in S$ and
$$b^3=c^2.$$
Also by Lemma~\ref{lemma:A-1} applied to $b,c$, we have
$$c^{-2}=(c^2)^b=(b^3)^b=b^3=c^2$$
and hence $c^4=1$. Moreover, from $bcb=c^{-1}$ and $bc^{-1}b=c$, we get
$$c^5=c^2\cdot c\cdot c^2=b^3cb^3=b^2(bcb)b^2=b^2(c^{-1})b^2=b(bc^{-1}b)b=bcb=c^{-1}$$
and hence $c^6=1$. Now $c^6=c^4=1$ yields $c^2=1$, contradicting $|c|>2$. Thus the second graph of Figure~\ref{Fig6} can never arise as $\Gamma[S]$, and in all cases $\Gamma$ is an ORR.~$_\blacksquare$

\smallskip

For the rest of the argument we may suppose that $|b|= 4$. 

\smallskip

\noindent\textsc{Case 2: }$|b|=4$, $|a|>4$ and $|ca^{-2}|>2$.

\smallskip

\noindent Consider $S=\{a,a^2,c,ca^{-2},b\}$ and $\Gamma=\Cay(G,S)$. 
As $|ca^{-2}|>2$ and $|a|>4$, $S\cap S^{-1}=\emptyset$ and $\Gamma=\Cay(G,S)$ is an oriented Cayley digraph.

Consider the subgraph $\Gamma[S]$. It is easy to verify that $(a,a^2)$, $(a^2,c)$ $(ca^{-2},c)$ are arcs of $\Gamma[S]$ (see Figure~\ref{Fig7}). As above, using the fact that $d(G)=3$, $|ab^{-1}|=|ab|=|bc|=|bc^{-1}|=2$, $[a,c]=1$ and $|b|=4$,  it is a routine computation to show that the only possible arcs of $\Gamma[S]$ (apart from the arcs $(a,a^2)$, $(a^2,c)$, $(ca^{-2},c)$) are $(c,a)$, $(ca^{-2},a)$ and $(ca^{-2},a^2)$. (See Figure~\ref{Fig7}.)
\begin{figure}[!hhh]
\begin{center}
\begin{tikzpicture}[node distance=1.5cm]
\node[circle,draw,inner sep=1pt, label=90:$a$](A0){};
\node[right=of A0,circle,draw,inner sep=1pt, label=90:$a^2$](A1){};
\node[below=of A1,circle,draw,inner sep=1pt,label=-90:$c$](A2){};
\node[left=of A2,circle,draw,inner sep=1pt,label=-90:$ca^{-2}$](A3){};
\node[right=of A1,circle,draw,inner sep=1pt,label=90:$b$](A4){};
\draw[-angle 45, ] (A0) to  (A1);
\draw[-angle 45, ] (A1) to  (A2);
\draw[-angle 45, ] (A3) to  (A2);
\draw[-angle 45, ,dashed] (A3) to  (A0);
\draw[-angle 45, ,dashed] (A3) to  (A1);
\draw[-angle 45, ,dashed] (A2) to (A0);
\end{tikzpicture}
\end{center}
\caption{Figure for the proof of Lemma~\ref{lemma:BB1}, Case 2}\label{Fig7}
\end{figure}
  This requires only routine computations and a tedious case-by-case analysis, we leave the details to the reader. Here we deal with only one particular case that in our opinion is the hardest. If $(b,a^2)$ is an arc of $\Gamma[S]$, then $a^2b^{-1}\in S$; now using the irredundancy of $\{a,b,c\}$, we deduce $a^2b^{-1}=b$ and hence $b^2=a^{2}$. As $|b|=4$, we get $a^4=1$, contradicting $|a|>4$. All other cases are similar.  

Moreover, using the irredundancy of $\{a,b,c\}$, one of the following holds: 
\begin{itemize}
\item $(c,a)$ is an arc of $\Gamma[S]$ and $a^3=c^2$,
\item $(ca^{-2},a)$ is an arc of $\Gamma[S]$ and either $a^3=c^2$ or $a^5=c^2$,
\item $(ca^{-2},a^{2})$ is an arc of $\Gamma[S]$ and either $a^6=c^2$ or $a^4=c^2$, or
\item $\Gamma[S]$ has only the arcs $(a,a^2)$, $(a^2,c)$, and $(ca^{-2},c)$.
\end{itemize}
If $a^3=c^2$, then $\Gamma[S]$ is isomorphic to the graph shown in Figure~\ref{Fig8}; in particular, it is easy to verify that $\Gamma[S]$ is asymmetric; thus by Lemma~\ref{Watkins-Nowitz} $\Gamma$ is an ORR. Suppose that $a^3\ne c^2$. In particular, $(c,a)$ is not an arc of $\Gamma[S]$ and $\Gamma[S]$ is isomorphic to one of the graphs in Figure~\ref{Fig9}. All graphs in Figure~\ref{Fig9} are asymmetric, therefore  by Lemma~\ref{Watkins-Nowitz} $\Gamma$ is an ORR.~$_\blacksquare$

\begin{figure}[!hhh]
\begin{center}
\begin{tikzpicture}[node distance=1.5cm]
\node[circle,draw,inner sep=1pt, label=90:$a$](A0){};
\node[right=of A0,circle,draw,inner sep=1pt, label=90:$a^2$](A1){};
\node[below=of A1,circle,draw,inner sep=1pt,label=-90:$c$](A2){};
\node[left=of A2,circle,draw,inner sep=1pt,label=-90:$ca^{-2}$](A3){};
\node[right=of A1,circle,draw,inner sep=1pt,label=90:$b$](A4){};
\draw[-angle 45, ] (A0) to  (A1);
\draw[-angle 45, ] (A1) to  (A2);
\draw[-angle 45, ] (A3) to  (A2);
\draw[-angle 45, ](A2) to (A0);
\draw[-angle 45, ](A3) to (A0);
\end{tikzpicture}
\end{center}
\caption{Figure for the proof of Lemma~\ref{lemma:BB1}, Case 2}\label{Fig8}
\end{figure}
 \begin{figure}
\begin{center}
\begin{tikzpicture}[node distance=.9cm]
\node[circle,draw,inner sep=1pt, label=90:$a$](A0){};
\node[right=of A0,circle,draw,inner sep=1pt, label=90:$a^2$](A1){};
\node[below=of A1,circle,draw,inner sep=1pt,label=-90:$c$](A2){};
\node[left=of A2,circle,draw,inner sep=1pt,label=-90:$ca^{-2}$](A3){};
\node[right=of A1,circle,draw,inner sep=1pt,label=90:$b$](A4){};
\draw[-angle 45, ] (A0) to  (A1);
\draw[-angle 45, ] (A1) to  (A2);
\draw[-angle 45, ] (A3) to  (A2);
\node[right=of A4,circle,draw,inner sep=1pt, label=90:$a$](AA0){};
\node[right=of AA0,circle,draw,inner sep=1pt, label=90:$a^2$](AA1){};
\node[below=of AA1,circle,draw,inner sep=1pt,label=-90:$c$](AA2){};
\node[left=of AA2,circle,draw,inner sep=1pt,label=-90:$ca^{-2}$](AA3){};
\node[right=of AA1,circle,draw,inner sep=1pt,label=90:$b$](AA4){};
\draw[-angle 45, ] (AA0) to  (AA1);
\draw[-angle 45, ] (AA1) to  (AA2);
\draw[-angle 45, ] (AA3) to  (AA2);
\draw[-angle 45, ](AA3) to (AA0);
\node[right=of AA4,circle,draw,inner sep=1pt, label=90:$a$](AAA0){};
\node[right=of AAA0,circle,draw,inner sep=1pt, label=90:$a^2$](AAA1){};
\node[below=of AAA1,circle,draw,inner sep=1pt,label=-90:$c$](AAA2){};
\node[left=of AAA2,circle,draw,inner sep=1pt,label=-90:$ca^{-2}$](AAA3){};
\node[right=of AAA1,circle,draw,inner sep=1pt,label=90:$b$](AAA4){};
\draw[-angle 45, ] (AAA0) to  (AAA1);
\draw[-angle 45, ] (AAA1) to  (AAA2);
\draw[-angle 45, ] (AAA3) to  (AAA2);
\draw[-angle 45, ](AAA3) to (AAA1);
\node[right=of AAA4,circle,draw,inner sep=1pt, label=90:$a$](AAAA0){};
\node[right=of AAAA0,circle,draw,inner sep=1pt, label=90:$a^2$](AAAA1){};
\node[below=of AAAA1,circle,draw,inner sep=1pt,label=-90:$c$](AAAA2){};
\node[left=of AAAA2,circle,draw,inner sep=1pt,label=-90:$ca^{-2}$](AAAA3){};
\node[right=of AAAA1,circle,draw,inner sep=1pt,label=90:$b$](AAAA4){};
\draw[-angle 45, ] (AAAA0) to  (AAAA1);
\draw[-angle 45, ] (AAAA1) to  (AAAA2);
\draw[-angle 45, ] (AAAA3) to  (AAAA2);
\draw[-angle 45, ](AAAA3) to (AAAA1);
\draw[-angle 45, ](AAAA3) to (AAAA0);
\end{tikzpicture}
\end{center}
\caption{Figure for the proof of Lemma~\ref{lemma:BB1}, Case 2}\label{Fig9}
\end{figure}

\smallskip

\noindent\textsc{Case 3: }$|b|=4$, $|c|>4$ and $|ac^{-2}|>2$.

\smallskip

\noindent The argument  is as in Case~$2$ using the connection set  $S=\{a,c,c^2,ac^{-2},b\}$ and $\Gamma=\Cay(G,S)$.~$_\blacksquare$

\smallskip

\noindent\textsc{Case 4: }$|b|=4$, $|ca^{-2}|=|ac^{-2}|=2$.

\smallskip

\noindent Now $G$ is a quotient of the group with presentation
\begin{equation*}
\begin{split}\langle x,y,z\mid &[x,z]=(xy)^2=(xy^{-1})^2=(yz)^2=(yz^{-1})^2=y^4\\
&=(xz^{-2})^2=(zx^{-2})^2=1\rangle.
\end{split}
\end{equation*}
A computation with \texttt{magma} \cite{magma} shows that this group has order $48$. Another computation with \texttt{magma} \cite{magma} shows 
 that each quotient $G$ of $P$ with $d(G)=3$ and with $G$ generated by three non-involutions $a,b,c$ with $[a,c]=1$, $|ca^{-1}|>2$, $|b|=4$ and $|ca^{-2}|=|ac^{-2}|=2$ admits an ORR.~$_\blacksquare$

\smallskip

\noindent\textsc{Case 5: }$|b|=4$, $|a|= 4$, $|c|>4$, $|ac^{-2}|=2$ and $|ca^{-2}|>2$.

\smallskip

\noindent Now $G$ is quotient of the group with presentation
\begin{equation*}
\begin{split}
P=\langle x,y,z\mid &[x,z]=(xy)^2=(xy^{-1})^2=(yz)^2=(yz^{-1})^2\\&=y^4=(xz^{-2})^2=x^4=1\rangle.
\end{split}
\end{equation*}
A computation in \texttt{magma} \cite{magma} shows that  $P$ has order $64$ and that each quotient $G$ of $P$ with $d(G)=3$ and with $G$ generated by three non-involutions $a,b,c$ with $[a,c]=1$, $|ca^{-1}|>2$, $|b|=4$, $|c|>4$, $|ac^{-2}|=2$ and $|ca^{-2}|>2$ admits an ORR.~$_\blacksquare$

\smallskip

\noindent\textsc{Case 6: }$|b|=4$, $|a|>4$, $|c|= 4$, $|ac^{-2}|>2$ and $|ca^{-2}|=2$.

\smallskip

\noindent The argument here is exactly as in Case~$5$ with the roles of $a$ and $c$ interchanged.~$_\blacksquare$

\smallskip

\noindent\textsc{Case 7: }$|b|=4$, $|a|= 4$ and $|c|= 4$.

\smallskip
 
\noindent In this case  $G$ is a quotient of the group $P$ with presentation
$$P=\langle x,y,z\mid [x,z]=(xy)^2=(xy^{-1})^2=(yz)^2=(yz^{-1})=y^4=x^4=z^4=1\rangle.$$
A computation with \texttt{magma}  \cite{magma} shows that this group has order $64$. Another computation with \texttt{magma}  \cite{magma} shows that each quotient $G$ of $P$ with $d(G)=3$ and with $G$ generated by three non-involutions $a,b,c$ with $[a,c]=1$, $|ca^{-1}|>2$, $|b|=|a|=|c|=4$ admits an ORR.~$_\blacksquare$
\end{proof}

To deal with the groups described in Lemma~\ref{lemma:2A}~(vi), we further subdivide the groups into two families. First we consider those for which $[b,c]=1$.

\begin{lemma}\label{ell=3 Q8 commute}
Let $G$ be a  group with $d(G)=3$ and with a generating set $\{a,b,c\}$ such that $|a|,|b|,|c|,|ba^{-1}|, |cb^{-1}|>2$, $[a,b]\neq 1$, and $[b,c]=1$. Then $G$ admits an ORR.
\end{lemma}

\begin{proof}
Let $S=\{a,b,c^{-1},ba^{-1},bc^{-1}\}$, and let $\Gamma=\Cay(G,S)$. Since $|a|, |b|,|c|$, $|ba^{-1}|, |cb^{-1}|>2$ and $bc^{-1}=(cb^{-1})^{-1}$, we see that $\Gamma$ is an oriented Cayley digraph. By Lemma~\ref{Q8-inside}, we may assume that $\langle a,b \rangle \cong Q_8$. Therefore $\Gamma[S]$ contains the arcs $(a,b)$, $(b,ba^{-1})$, and $(ba^{-1},a)$. Furthermore, since $[b,c]=1$, $\Gamma[S]$ contains the arcs $(b,bc^{-1})$ and $(c^{-1},bc^{-1})$. Straightforward calculations using the assumptions that $d(G)=3$ and that $|b|, |c|>2$ show that there are no other arcs in $\Gamma[S]$: see Figure~\ref{CALGARY}.
\begin{figure}[!hhh]
\begin{center}
\begin{tikzpicture}[node distance=1.5cm]
\node[circle,draw,inner sep=1pt, label=90:$a$](A0){};
\node[right=of A0,circle,draw,inner sep=1pt, label=90:$b$](A1){};
\node[right=of A1,circle,draw,inner sep=1pt,label=90:$c^{-1}$](A2){};
\node[below=of A0,circle,draw,inner sep=1pt,label=-90:$ba^{-1}$](A3){};
\node[below=of A1,circle,draw,inner sep=1pt,label=-90:$bc^{-1}$](A4){};
\draw[-angle 45, ] (A0) to  (A1);
\draw[-angle 45, ] (A1) to  (A3);
\draw[-angle 45, ] (A3) to  (A0);
\draw[-angle 45, ] (A1) to  (A4);
\draw[-angle 45, ] (A2) to  (A4);
\end{tikzpicture}
\end{center}
\caption{Figure for the proof of Lemma~\ref{ell=3 Q8 commute}}\label{CALGARY}
\end{figure}
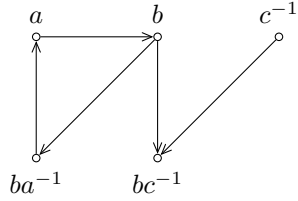
 
   Now we see that $\Gamma[S]$ is asymmetric, so by Lemma~\ref{Watkins-Nowitz}, $\Gamma$ is an ORR for $G$.
\end{proof}

Finally, we deal with the groups described in Lemma~\ref{lemma:2A}~(vi) that have $[b,c] \neq 1$. The next proof is distinct from most of the arguments in this paper, since examining the induced subgraph on the neighbours of $1$ is not sufficient to show that our putative ORR is in fact an ORR.

\begin{lemma}\label{ell=3 Q8 Q8}
Let $G$ be a  group with $d(G)=3$ and with a generating set $\{a,b,c\}$ such that $|a|,|b|,|c|,|ba^{-1}|, |cb^{-1}|>2$, $[a,b]\neq 1$, and $[b,c] \neq 1$. Then $G$ admits an ORR.
\end{lemma}

\begin{proof}
By Lemma~\ref{Q8-inside}, we may assume that $\langle a,b \rangle \cong Q_8$. Also, by applying Lemma~\ref{Q8-inside} to $\{c,b,a\}$ (we can do this since $|bc^{-1}|=|cb^{-1}|>2$ and $|ab^{-1}|=|ba^{-1}|>2$), we may assume that $\langle c,b \rangle \cong Q_8$. In particular, $$b^2=c^2=a^2=(ab)^2.$$

Let $S=\{a,b,c,ab\}$. Since $\langle a,b \rangle \cong Q_8$, $|ab|=4$ so (using also $|a|, |b|, |c|>2$ and $d(G)=3$) we have $S \cap S^{-1}=\emptyset$ and $\Gamma=\Cay(G,S)$ is an oriented Cayley digraph. Also, since $\langle a, b \rangle \cong Q_8$, we see that $\Gamma[S]$ has arcs $(b,ab)$, $(ab,a)$ (since $a=bab$), and $(a,b)$ (since $b=aba$). Since $d(G)=3$, $c$ is an isolated vertex of $\Gamma[S]$: see Figure~\ref{NEWFIG}. Thus, the only non-identity automorphisms of $\Gamma[S]$ fix $c$ and act as either the $3$-cycle $(a\ b \ ab)$ or the $3$-cycle  $(a \ ab\ b)$ on $\{a,b,ab\}$.
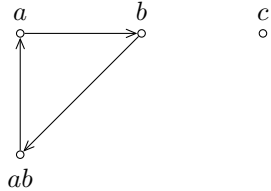
\begin{figure}[!hhh]
\begin{center}
\begin{tikzpicture}[node distance=1.5cm]
\node[circle,draw,inner sep=1pt, label=90:$a$](A0){};
\node[right=of A0,circle,draw,inner sep=1pt, label=90:$b$](A1){};
\node[below=of A0,circle,draw,inner sep=1pt,label=-90:$ab$](A2){};
\node[right=of A1,circle,draw,inner sep=1pt,label=90:$c$](A3){};
\draw[-angle 45, ] (A0) to  (A1);
\draw[-angle 45, ] (A1) to  (A2);
\draw[-angle 45, ] (A2) to  (A0);
\end{tikzpicture}
\end{center}
\caption{Figure for the proof of Lemma~\ref{ell=3 Q8 Q8}}\label{NEWFIG}
\end{figure}

Let $\varphi$ be an arbitrary automorphism in $\Aut(\Gamma)_1$. Since $\varphi$ fixes $S$ setwise, $\varphi\vert_S$ induces an automorphism of $\Gamma[S]$. If every such $\varphi$ acts trivially on $S$, then by Lemma~\ref{Watkins-Nowitz}, $\Gamma$ is an ORR for $G$ and we are done. Suppose then that there exists $\varphi\in \Aut(\Gamma)_1$ acting non-trivially on $S$. Replacing $\varphi$ by $\varphi^{-1}$ if necessary, we may assume without loss of generality that 
\begin{eqnarray}\label{Rome}
\varphi(a)&=&b,\\\nonumber
 \varphi(b)&=&ab,\\\nonumber
 \varphi(ab)&=&a.
 \end{eqnarray}

Given a vertex $x$ of $\Gamma$, we denote by $\Gamma^+(x)$ the out-neighbourhood of $x$, that is, $\Gamma^+(x)=\{ax,bx,cx,abx\}$. An easy computation, using $\langle a,b\rangle\cong Q_8\cong \langle b,c\rangle$ and $a^2=b^2=c^2$, yields
\begin{eqnarray*}
\Gamma^+(a)&=&\{a^2,ba,ca,b\},\\
\Gamma^+(b)&=&\{a^2,ab,cb,a^{-1}\},\\
\Gamma^+(ab)&=&\{a^2,b^{-1},a,cab\},\\
\Gamma^+(c)&=&\{a^2,ac,bc,abc\}.
\end{eqnarray*}
Since $d(G)=3$ and $[a,b]\ne 1$, it is easy to verify that $a^2$ is the unique out-neighbour in common to $a$ and $b$.  Moreover, $a^2$ is the unique mutual out-neighbour in common to all four vertices of $S$. That is,
$$\Gamma^+(a)\cap \Gamma^+(b)=\{a^2\}\,\,\textrm{ and }\,\,\Gamma^+(a)\cap \Gamma^+(b)\cap \Gamma^+(ab)\cap \Gamma^+(c)=\{a^2\}.$$ Therefore $a^2$ is fixed by $\varphi$; furthermore this implies that, for each vertex $x$ of $\Gamma$,  
\begin{equation}\label{Florence}\varphi(xa^2)=\varphi(x)a^2
\end{equation} because $xa^2$ is the unique mutual out-neighbour of the four out-neighbours of $x$, and $\varphi(x)a^2$ has the same property with respect to $\varphi(x)$.

Eq.~\eqref{Rome} gives 
$$\varphi(\{a^2,ab,cb,a^{-1}\})=\varphi(\Gamma^+(b))=\Gamma^+(\varphi(b))=\Gamma^+(ab)=\{a^2,b^{-1},a,cab\}.$$
Now, Eq.~\eqref{Rome} gives $\varphi(ab)=a$ and, Eq.~\eqref{Florence} applied first with $x=1$ and then with $x=a$ gives $\varphi(a^2)=\varphi(1)a^2=a^2$ and $\varphi(a^{-1})=\varphi(a\cdot a^2)=\varphi(a)a^2=ba^2=b^{-1}$. Then we must have 
$$\varphi(cb)=cab.$$

Finally, Eq.~\eqref{Rome} gives
$$\varphi(\{a^2,ac,bc,abc\})=\varphi(\Gamma^+(c))=\Gamma^+(\varphi(c))=\Gamma^+(c)=\{a^2,ac,bc,abc\}.$$
As $\varphi(a^2)=a^2$, we get
$$\varphi(\{ac,bc,abc\})=\{ac,bc,abc\}.$$
In particular, using $\langle a,b\rangle\cong Q_8$ and Eq.~\eqref{Florence}, we get $$\varphi(bc)=\varphi(cb^{-1})=\varphi(cba^2)=\varphi(cb)a^2=caba^2=cab^{-1} \in \{ac, bc, abc\}.$$ Both $cab^{-1}=ac$ and $cab^{-1}=bc$ contradict $d(G)=3$, so we must have $$cab^{-1}=abc=acb^{-1},$$ where in the last equality we used $\langle b, c \rangle \cong Q_8$. This implies $[a,c]=1$. But then $a$ and $c$ have a mutual out-neighbour in addition to $a^2$, namely $ac=ca$, so that $b=\varphi(a)$ and $c=\varphi(c)$ must have a mutual out-neighbour in addition to $a^2$. Thus $\{ab,cb,a^{-1}\}\cap \{ac,bc,abc\} \neq \emptyset$. However the only way this can happen with $d(G)=3$ is if $cb=bc$, contradicting $\langle b,c \rangle\cong Q_8$.
\end{proof}

We point out that the arguments in the proof of Lemma~\ref{ell=3 Q8 Q8} have some similarities with the proof of the main theorem in~\cite{SpigaCI}.

We are now in a position to summarize all of our results on  groups $G$ with $d(G)=3$ in a complete classification.

\begin{theorem}\label{thrm:A2}
Let $G$ be a finite group with $d(G)=3$. Then one of the following holds:
\begin{enumerate}
\item[(i)]$G$ is generalized dihedral;
\item[(ii)]$G$ admits an ORR;
\item[(iii)]$G$ is isomorphic to one of the following groups:
\begin{itemize}
\item $C_4\times C_2^2$, 
\item $C_3\times C_2^3$, 
\item
$\begin{aligned}[t]\langle x,y,z\mid &x^4=y^4=z^4=(xy)^2=(xy^{-1})^2=1,\\
&(xz)^2=(xz^{-1})^2=(yz)^2=(yz^{-1})^2=x^2y^2z^2=1\rangle,\end{aligned}$
\item
$\begin{aligned}[t]\langle x,y,z\mid &x^4=y^4=z^4=(yx)^2=(yx^{-1})^2=(yz)^2=(yz^{-1})^2=1,\\
&x^2=z^2,x^z=x^{-1}, x^2=y^2\rangle.
\end{aligned}$
\end{itemize}
\end{enumerate}
Furthermore, the groups in (i) and (iii) admit no ORR.
\end{theorem}
\begin{proof}
We follow the subdivision in Lemma~\ref{lemma:2A}. Suppose that $G$ is not generalized dihedral and that $G$ admits no ORR. 

If $G$ is abelian and admits a generating set $\{a,b,c\}$ with $|a|,|b|,|c|,|ba^{-1}|$, $|cb^{-1}|>2$ (that is, $G$ is as in part (ii) of Lemma~\ref{lemma:2A}), then $G$ is isomorphic to $C_3\times C_2^3$ by Lemma~\ref{abelian-3-gen}, so falls into part~(iii) of this theorem. 
If $G$ is as in Lemma~\ref{lemma:2A}~(iii), then by Lemma~\ref{lemma:6}
$G$ is isomorphic to $C_4 \times C_2^2$ or to
\begin{align*}
\langle x,y,z\mid &x^4=y^4=z^4=(xy)^2=(xy^{-1})^2=1,\\
&(xz)^2=(xz^{-1})^2=(yz)^2=(yz^{-1})^2=x^2y^2z^2=1\rangle,
\end{align*} and both of these groups appear in part~(iii) of this theorem. 
If $G$ admits a generating set as in Lemma~\ref{lemma:2A}~(iv), then by Lemma~\ref{lemma:B1} we have
\begin{align*}G\cong\langle x,y,z\mid &x^4=y^4=z^4=(yx)^2=(yx^{-1})^2=(yz)^2=(yz^{-1})^2=1,\\
&x^2=z^2,x^z=x^{-1}, x^2=y^2\rangle,
\end{align*}
the final group listed in part~(iii) of this theorem. 

If $G$ admits a generating set as in Lemma~\ref{lemma:2A}~(v), then $G$ always admits an ORR by Lemma~\ref{lemma:BB1}. If $G$ admits a generating set as in Lemma~\ref{lemma:2A}~(vi) and $[b,c]=1$ then $G$ admits an ORR by Lemma~\ref{ell=3 Q8 commute}, while if $[b,c]\neq 1$ then $G$ admits an ORR by Lemma~\ref{ell=3 Q8 Q8}. This completes the classification.

We have previously observed that generalized dihedral groups cannot admit ORRs, since they admit no generating sets that avoid elements of order $2$. We can use \texttt{magma}~\cite{magma} to show that the four groups listed in (iii) admit no ORRs.
\end{proof}

\section{Five-product-avoiding generating sets with a useful ordering}\label{JoyLemma}

The goal of this section is to prove that with a few exceptions of small cardinality, if a group $G$ admits a generating set that is largely irredundant, none of whose elements are involutions, and this generating set can be ordered so that no product $ba^{-1}$ is an involution, where $b$ is the element that immediately follows $a$ in the ordering, then $G$ admits an ORR. To this end, the first thing we need to do is explain what we mean by ``largely irredundant." The following definition gives a weak form of near-irredundancy  that will be required in the proofs of the results that follow.

\begin{definition}\label{5pg}
{\rm Let $G$ be a group with generating set $T$. If $T$ has the property that 
\begin{itemize}
\item $T\cap T^{-1}=\emptyset$ and 
\item for any $x,y,z,w,v \in T \cup T^{-1}\cup \{1\}$, $$xyzwv \not\in T\setminus\{x^{-1},x,y^{-1},y,z^{-1},z,w^{-1},w,v^{-1},v\},$$
\end{itemize} then we say that $T$ is a {\em five-product-avoiding generating set} for $G$.}
\end{definition}

Clearly, irredundant generating sets are five-product-avoiding and hence, in turn, generating sets of minimum cardinality are five-product-avoiding.

With this definition in hand, we turn to a lengthy result that proves a lot of useful facts about a few particular Cayley graphs whose connection sets are based on five-product-avoiding generating sets that admit the type of ordering we want.

In most of the results in this section, we will only consider generating sets with at least four elements, since the cases $d(G)=2$ and $d(G)=3$ are classified above.

\begin{proposition}\label{usefulfacts}
Let $T=\{a_1, \ldots, a_\ell\}$ be a five-product-avoiding generating set for $\langle T \rangle$, with $\ell\ge 4$.
Let $X=\{a_{i+1}a_{i}^{-1}\mid 1 \le i \le \ell-1\}$. Suppose that $|s|>2$ for every $s \in T \cup X$. 

If $a_1$ and $a_2$ commute and $|a_1a_2|>2$, let $a_0$ be $a_1a_2$. Let $S$ be any one of $T\cup X$, $T\cup X\cup\{a_0\}$, or $T\cup X\cup \{a',a_1a'\}$, where $a'\notin T$ commutes with $a_1$, $|a'|>2$, $|a_1a'|>2$, and $T'=\{a'\}\cup T$ is a five-product-avoiding generating set for $\langle T' \rangle$. Let $\Gamma=\Cay(\langle S \rangle, S)$.

Then
\begin{enumerate}
\item there is no duplication among $X$, $T$, $a_0$, $a'$, and $a_1a'$;
\item $S \cap S^{-1}=\emptyset$;
\item if $x, y \in S$ and there is an arc from $x$ to $y$ in $\Gamma$, then one of the following holds:
\begin{itemize}
\item $\{x, y, yx^{-1}\} \subseteq \{a_i, a_{i+1}, a_{i+1}a_i^{-1}\}$ for some $1 \le i \le \ell-1$;
\item $a_0 \in S$, and either $y=a_0$ and $\{x, yx^{-1}\} =\{a_1, a_2\}$, or 
 $y=a_2a_1^{-1}$, $\{x,yx^{-1}\}=\{a_1,a_0\}$, and $|a_1|=3$; or
\item $a', a_1a' \in S$, $\{x,yx^{-1}\}=\{a_1,a'\}$ and $y=a_1a'$;
\end{itemize}
\item for $1 \le i \le \ell-1$, $\Gamma$ has
\begin{itemize}
\item an arc from $a_{i+1}a_{i}^{-1}$ to $a_{i+1}$ if and only if $a_i$ and $a_{i+1}$ commute;
\item an arc from $a_{i+1}$ to $a_{i+1} a_{i}^{-1}$ if and only if $a_{i+1}$ inverts $a_{i}$; and
\item no arc from $a_{i}$ to $a_{i+1}a_{i}^{-1}$, unless $i=1$, $a_0 \in S$, and $|a_1|=3$;
\end{itemize}
\item the induced subgraph $\Gamma[T]$ is a directed path $a_1, \ldots, a_\ell$ of length $\ell-1$; 
\item removing the endpoints of any directed induced path of length $k \ge \ell-1$ in $\Gamma[S]$ leaves a subpath of the directed path given in (5);
\item if $a_0 \in S$ and $\Aut(\Gamma)_1$ fixes every vertex of $\{a_2, \ldots, a_\ell\}$, then $\Aut(\Gamma)_1$ must fix $a_0$;
\item if $\Aut(\Gamma)_1$ fixes $a_2, \ldots, a_{\ell-1}$, then $\Aut(\Gamma)_1$ must also fix $a_\ell$;
\item if $a', a_1a' \in S$, and $\Aut(\Gamma)_1$ fixes every vertex of $S$ except possibly $a_1$, $a_2a_1^{-1}$, $a'$, and $a_1a'$, then it fixes every vertex of $S$;
\item $\Aut(\Gamma)_1$ must fix every vertex of $S$ except possibly $a_1$ and $a_2a_1^{-1}$.
\end{enumerate}
\end{proposition}

\begin{proof}[Proof of~$(1)$]
Since $T$ is five-product-avoiding, it is clear that $X\cap T=\emptyset$, and $a_0=a_1a_2 \not\in T$. Since $T'$ is five-product-avoiding, we also cannot have $a' \in X \cup T$, or $a_1a'\in X \cup T$. Finally, $a_0 \in X$ contradicts $T$ being five-product-avoiding unless $a_0=a_1a_2=a_2a_1^{-1}$, but since $a_0$ exists only if $[a_1,a_2]=1$, this implies $a_1^2=1$, a contradiction.
 \end{proof}
 
\begin{proof}[Proof of~$(2)$]
 If some $x \in S \cap S^{-1}$, then $x, x^{-1} \in S$.  If $x=a'$, then since $|a'|>2$ implies $x^{-1} \neq a'$, $x^{-1} \in S$ yields $a'^{-1}\in T\cup X\cup\{a_0\}\cup \{a_1a'\}$. This contradicts $T'$ being five-product-avoiding. Thus, we may assume henceforth that $x, x^{-1} \neq a'$.
 
Similarly, if $x=a_1a'$, then since $|a_1a'|>2$ implies $x^{-1} \neq a_1a'$, $x^{-1} \in S$ yields $(a_1a')^{-1} \in T \cup X \cup \{a_0\}\cup\{a'\}$. This contradicts $T'$ being five-product-avoiding. Thus, we may assume henceforth that $x, x^{-1} \neq a_1a'$.
 
 If $x=a_i \in T$, then $x^{-1}\in S$ and $x^{-1}\neq a',a_1a'$ yields $a_i^{-1}\in T \cup X \cup \{a_0\}$. If $a_i^{-1}=a_j \in T$ or $a_i^{-1}=a_{j+1}a_j^{-1}$ and $j \neq i$, then this contradicts $T$ being five-product-avoiding, while if $j=i$ then this contradicts $|a_i|>2$ or $|a_{j+1}|>2$. If $a_i^{-1}=a_0$, then this again contradicts $T$ being five-product-avoiding. We may henceforth assume $x, x^{-1} \in X \cup \{a_1a_2\}$. Furthermore, if $a_0=a_1a_2 \in S$, then $|a_0|>2$, so we cannot have $x=x^{-1}=a_0$. We may therefore assume $x \in X$.
 
Let $x=a_{i+1}a_{i}^{-1} \in X$. Then $x^{-1}\in S$ and $x^{-1} \neq a',a_1a'$ yields $a_{i}a_{i+1}^{-1}\in T \cup X \cup\{a_0\}$. This contradicts $T'$ being five-product-avoiding unless $x^{-1}=a_ia_{i+1}^{-1}=a_{i+1}a_i^{-1}$, or $x^{-1}=a_1a_2^{-1}=a_1a_2$. The first of these contradicts $|s|>2$ for every $s \in X$, while the second gives $a_2^2=1$, contradicting the same hypothesis.

\end{proof} 

\begin{proof}[Proof of~$(3)$]
 Our construction of elements of $S$ shows that each of $y$, $x$, and $yx^{-1}$ can be written as a product of at most two elements of $T'\cup (T')^{-1}$: say $b_1,b_3,b_5 \in T'\cup\{1\}$ and $b_2,b_4,b_6\in T^{-1}\cup\{a_2,a',1\}$ with $b_1b_2=y$, $b_3b_4=x$, and $b_5b_6=yx^{-1}$. Thus, $y=(yx^{-1})x$ gives the equation $b_1b_2=b_5b_6b_3b_4$. Define an equivalence relation on the subscripts $\{1, \ldots, 6\}$ by $i \equiv j$ if and only if $b_i =b_j^{\pm1}$. The trivial equivalence class contains those $i$ such that $b_i=1$.
 
We observe that we cannot have $y=x$ or $y=yx^{-1}$, since these would imply $yx^{-1}=1$ and $x=x^{-1}=1$, respectively. This contradicts $x, y, yx^{-1} \in S$, since every element of $S$ has order greater than $2$.
 
Every nontrivial equivalence class has cardinality at least $2$. Otherwise, say $i$ is in an equivalence class of cardinality $1$. Then we can rearrange the equation so that $b_i$
 is written as a product of five elements of $(T' \cup (T')^{-1} \cup \{1\})\setminus\{b_i,b_i^{-1}\}$, contradicting the assumption that $T'$ is five-product-avoiding. 

There are at least two nontrivial equivalence classes. By the format of elements of $S$, if there were only one nontrivial equivalence class then $y, x, yx^{-1}\in T'$ must all be equal, contradicting our earlier observation.

Suppose that there are three nontrivial equivalence classes, so each has cardinality $2$. Observe that for each odd $t$, we must have $\{b_{t},b_{t+1}\}=\{a_r,a_{r+1}\}$ for some $r$, or $\{a_1,a'\}$. It is not possible to choose three pairs of this sort whose union is three elements, with each of the three elements appearing in two of the pairs, so this possibility cannot occur.  

We may therefore suppose that there are exactly two nontrivial equivalence classes, each with cardinality at least two. 

First consider the possibility that  the trivial equivalence class is empty, so there exist $i,j \in \{1, \ldots, 6\}$ such that for every $1\le k \le 6$, we have $b_k \in \{b_i^{\pm1},b_j^{\pm1}\}$. Since $y \not\in \{x, yx^{-1}\}$, the only remaining possibility is $x=yx^{-1}$, $a_0\in S$ so $a_1$ and $a_2$ commute and $\{x, y, yx^{-1}\}=\{a_0, a_2a_1^{-1}\}$, but calculations show that this would contradict the assumption that $T$ is five-product-avoiding. 

The next possibility is that the trivial equivalence class contains a single element, 
so there are two nontrivial equivalence classes, $\{i_1,i_2,i_3\}$ and $\{j_1,j_2\}$.  If $b_{i_1}\not\in\{a_1^{\pm 1},a_2^{\pm1}\}$, then since we cannot have $y=x$ or $y=yx^{-1}$, we must have $x,yx^{-1}=a_{i+1}a_i^{-1}$, and $y= a_i$ or $a_{i+1}$ (either of these is a possibility that we have listed), or $x,yx^{-1}=a_1a'$ and $y=a'$, but this contradicts the assumption that $T'$ is five-product-avoiding. If $b_{i_1}=a_2^{\pm 1}$, then we either have $x,yx^{-1}=a_2a_1^{-1}$ and $y=a_2$ (which we have listed as a possibility), or $a_0 \in S$ so that $a_1$ and $a_2$ commute. In this case (since $y \neq x, yx^{-1}$), we either have $a_0^2=a_2$, or $\{x,y,yx^{-1}\}=\{a_0,a_2a_1^{-1},a_2\}$, both of which contradict the assumption that $T$ is five-product-avoiding or the assumption that $|a_2|>2$. If $b_{i_1}=a_1^{\pm1}$, then we either have $b_{j_1}=a'$ so that $x,yx^{-1}=a_1a'$ and $y=a_1$, but this contradicts the assumption that $T'$ is five-product-avoiding, or $b_{j_1}=a_2^{\pm1}$. In this case, we either have $x,yx^{-1}=a_2a_1^{-1}$ and $y=a_1$ (which we have listed as a possibility), or $a_0 \in S$ so that $a_1$ and $a_2$ commute. Now (since $y \neq x, yx^{-1}$), we either have $a_0^2=a_1$, or $\{x,y,yx^{-1}\}=\{a_0,a_2a_1^{-1},a_2\}$, both of which contradict the assumption that $T$ is five-product-avoiding, unless $|a_1|=3$, $y=a_2a_1^{-1}$, and $\{x,yx^{-1}\}=\{a_1,a_0\}$. 

The last possibility is that the trivial equivalence class contains exactly two elements. In this case, we see that two of $x,y,yx^{-1}$ are in fact elements of $T'$. We cannot have $x=yx^{-1}\in T'$ since this would force $1 \equiv 2$ which is impossible. So we must have two distinct elements of $T'$ in $\{x,y,yx^{-1}\}$. If $a_i,a_j \in \{x,y,yx^{-1}\}$ with $i < j$ then the structure of elements of $S$ implies that $j=i+1$ and the final element is either $a_ja_i^{-1}=a_{i+1}a_{i}^{-1}$, a possibility that we have listed, or $a_ia_j=a_1a_2=y$, also listed. Finally, we may have $\{x,yx^{-1}\}=\{a_1,a'\}$ and $y=a_1a'$. (There cannot be any arcs from $a_1a_2$ to $a_1$ or $a_2$ by (2), and similarly there cannot be arcs from $a_1a'$ to $a_1$ or $a'$; also, since $|a_1|>2$ and $T'$ is five-product-avoiding, there cannot be an arc between $a_1$ and $a'$.)
\end{proof}

\begin{proof}[Proof of~$(4)$]
 Let $1 \le i \le \ell-1$. It is easy to see that if $a_i$ and $a_{i+1}$ commute, then there is an arc from $a_{i+1}a_{i}^{-1}$ to $a_{i+1}$. On the other hand, if such an arc exists then we must have $a_{i
 +1}(a_{i}a_{i+1}^{-1}) \in S$, and in particular by (3) $a_{i+1} a_{i}a_{i+1}^{-1} \in \{a_i, a_{i+1}a_{i}^{-1},a_{i+1}\}$. In the third case we would have $a_i=a_{i+1}$ and in the second case we would have $a_{i+1}=a_{i}^{2}$; each of these contradicts the assumption that $T$ is five-product-avoiding. The only remaining possibility is $a_{i+1}a_{i}a_{i+1}^{-1}=a_{i}$; that is, $a_i$ and $a_{i+1}$ commute.

It is easy to see that if $a_{i+1}$ inverts $a_{i}$, then  there is an arc from $a_{i+1}$ to $a_{i+1} a_{i}^{-1}$. On the other hand, if such an arc exists, then we must have $a_{i+1}a_{i}^{-1}a_{i+1}^{-1} \in S$, and in particular by (3) $a_{i+1} a_{i}^{-1}a_{i+1}^{-1} \in \{a_i, a_{i+1}a_{i}^{-1}, a_{i+1}\}$. In the third case we would have $a_{i+1}=a_i^{-1}$, and in the second case we would have $a_{i+1}=1$; the first of these contradicts the assumption that $T$ is five-product-avoiding, while the second contradicts the assumption that every element of $T$ has order greater than 2. The only remaining possibility is $a_{i+1}a_{i}^{-1}a_{i+1}^{-1}=a_{i}$; that is, $a_{i+1}$ inverts $a_{i}$.

If there were an arc from $a_{i}$ to $a_{i+1} a_{i}^{-1}$ and we are not in the case $i=1$, $a_0\in S$, and $|a_1|=3$, then we must have $a_{i+1} a_{i}^{-2} \in S$, and in particular by (3), $a_{i+1} a_{i}^{-2} \in \{a_i, a_{i+1}, a_{i+1}a_{i}^{-1}\}$. In the second case, we would have $a_{i}^2=1$, and in the third case we would have $a_{i}=1$, each of which contradicts the assumption that every element of $T$ has order greater than $2$. In the first case, we would have $a_{i}^3=a_{i+1}$, contradicting the assumption that $T$ is five-product-avoiding. 
\end{proof}

\begin{proof}[Proof of~$(5)$]
 Since $(a_{i+1}a_i^{-1})a_i=a_{i+1}$, there is an arc from $a_i$ to $a_{i+1}$ for every $1 \le i \le \ell-1$. So the directed path exists. We need only show that there is no arc from $a_i$ to $a_j$ unless $j=i+1$. This is a straightforward consequence of (2) and (3).
 \end{proof}

\begin{proof}[Proof of~$(6)$]
By (3), $a_1a'$ has no out-neighbours, and $a'$ has no in-neighbours, so neither of these can be an interior vertex of a directed induced path in $\Gamma[S]$.

Observe using~(3) that the only possible neighbours of any vertex $a_{i+1}a_i^{-1}$ of $X$ are $a_i$ and $a_{i+1}$, except when $i=1$ in which case $a_0$, $a_1$ and $a_2$ are the only possible neighbours of $a_2a_1^{-1}$. By (5), there is an arc from $a_i$ to $a_{i+1}$, and if $a_0\in S$ then there are arcs from $a_1$ and $a_2$ to $a_0$. Thus, any two neighbours of $a_{i+1}a_i^{-1}$ are adjacent, and hence $a_{i+1}a_i^{-1}$ cannot be an interior vertex of any directed induced path in $\Gamma[S]$.

If $a_0$ is an interior vertex of some directed induced path of length $k$, then by (3) $a_0$ must be followed by $a_2a_1^{-1}$ and $|a_1|=3$. Since $a_2a_1^{-1}$ cannot be an interior vertex of an induced directed path, it must be the final vertex of this path. The path cannot include $a_2$ as the in-neighbour of $a_0$, since $a_0\in S$ implies that $a_1$ commutes with $a_2$, so there is an arc from $a_2a_1^{-1}$ to $a_2$ by (4). Thus, the in-neighbour of $a_0$ must be $a_1$. But $|a_1|=3$ implies that there is an arc from $a_1$ to $a_2a_1^{-1}=a_0a_1$, so the path is not induced. Thus $a_0$ is not an interior vertex of any directed induced path of length $k$. 

We conclude that any directed induced path of length $k \ge\ell-1$ contains $k-1$ interior vertices, all of which must lie in $T$. By (5), these vertices must induce a subpath of the directed path found in (5). 
\end{proof}

\begin{proof}[Proof of~$(7)$]
Observe that $\Aut(\Gamma)_1$ fixes the out-neighbours of $1$ setwise, so fixes $S$. Thus, it induces an automorphism of $\Gamma[S]$.

By assumption, $\Aut(\Gamma)_1$ fixes $a_2, \ldots, a_\ell$, so by Lemma~\ref{Watkins-Nowitz} the only vertices of $S$ that can be moved by $\Aut(\Gamma)_1$ are $a_1$, $a_2a_1^{-1}$, and $a_0$.
But since $a_1$ and $a_2$ commute, by (4) $a_0$ is the only one of these three vertices that is an out-neighbour of $a_2$, so is fixed, completing the proof.
\end{proof}

\begin{proof}[Proof of~$(8)$]
Observe that $\Aut(\Gamma)_1$ fixes the out-neighbours of $1$ setwise, so fixes $S$. Thus, it induces an automorphism of $\Gamma[S]$.

By assumption, $\Aut(\Gamma)_1$ fixes $a_2, \ldots, a_{\ell-1}$, and $\ell-1\ge 3$, so using Lemma~\ref{Watkins-Nowitz} it also fixes $a_{\ell-1} a_{\ell-2}^{-1}$. By (5), $a_\ell$ is an out-neighbour of $a_{\ell-1}$. By (3) and (4), the only other possible out-neighbour of $a_{\ell-1}$ is $a_{\ell-1}a_{\ell-2}^{-1}$, which is fixed by $\Aut(\Gamma)_1$, so $a_\ell$ must also be fixed by $\Aut(\Gamma)_1$.
\end{proof}

\begin{proof}[Proof of~$(9)$]
Observe that $\Aut(\Gamma)_1$ fixes the out-neighbours of $1$ setwise, so fixes $S$. Thus, it induces an automorphism of $\Gamma[S]$.

Suppose that $\Aut(\Gamma)_1$ fixes every vertex of $S$ except possibly $a_1$, $a_2a_1^{-1}$, $a_1a'$, and $a'$. Since $a_2$ is fixed by $\Gamma$ and $a_1$ is an in-neighbour of $a_2$ but $a_1a'$ and $a'$ are not (see (3)), the orbit of $\Aut(\Gamma)_1$ that contains $a_1$ can only contain $a_1$ and possibly $a_2a_1^{-1}$. If $a_2a_1^{-1}$ is in this orbit, then since there is an arc from $a_1$ to $a_1a'$, there must also be an arc from $a_2a_1^{-1}$ to either $a'$ or $a_1a'$, but by (3) this is not the case. Thus in any case, $\Aut(\Gamma)_1$ fixes $a_1$. Now by Lemma~\ref{Watkins-Nowitz}, $\Aut(\Gamma)_1$ fixes $a_2a_1^{-1}$. Of the remaining two vertices, $a_1a'$ is the unique out-neighbour of $a_1$, so both it and $a'$ are also fixed by $\Aut(\Gamma)_1$.
\end{proof}

\begin{proof}[Proof of~$(10)$] Observe that $\Aut(\Gamma)_1$ fixes the out-neighbours of $1$ setwise, so fixes $S$. Thus, it induces an automorphism of $\Gamma[S]$.
 Let $k$ be the length of a longest induced directed path in $\Gamma[S]$.
 By (5), $k \ge \ell-1$, so $k\ge 3$.
We can deduce from (6) that $k \le \ell+1$. 

By (6), every directed induced path of length $k$ includes at least $k-1$ consecutive vertices from $\{a_1, \ldots, a_\ell\}$.

If $k=\ell+1$, then $a_1, \ldots, a_\ell$ is an interior subpath of every directed induced path of length $k$. Thus, $a_1, \ldots, a_\ell$ are all fixed by $\Aut(\Gamma)_1$. By Lemma~\ref{Watkins-Nowitz}, this means that every vertex of $\langle T \rangle$ is fixed by $\Aut(\Gamma)_1$. By (9), every vertex of $S$ is fixed by $\Aut(\Gamma)_1$.

If $k=\ell$, then every directed induced path of length $k$ has either $a_1, \ldots, a_{\ell-1}$ or $a_2, \ldots, a_\ell$ as its interior vertices. In the first case, the path must begin with $a_2a_1^{-1}$ as this is the only possible in-neighbour of $a_1$. In the second case, the path must end with $a_\ell a_{\ell-1}^{-1}$ as this is the only possible out-neighbour of $a_\ell$. We cannot have induced paths of length $k$ that fall into both of these categories, because then $a_2a_1^{-1}, a_1, \ldots, a_\ell, a_{\ell} a_{\ell-1}^{-1}$ would be a longer induced directed path.  So either $a_1, \ldots, a_{\ell-1}$ or $a_2, \ldots, a_\ell$ are uniquely determined as the interior vertices (in that order) of every longest induced directed path. Thus, $\Aut(\Gamma)_1$ fixes all of the vertices $a_2,\ldots, a_{\ell-1}$, and either $a_1$ or $a_\ell$. By Lemma~\ref{Watkins-Nowitz}, this means that the only vertices of $S$ that $\Aut(\Gamma)_1$ can move are $a_1$, $a_2a_1^{-1}$, $a_0$, $a_1a'$, and $a'$, or $a_\ell$, $a_\ell a_{\ell-1}^{-1}$, $a_1a'$ and $a'$. Recall that if $a_0 \in S$ then $a_1a',a' \not\in S$ and vice versa.
In the first case, if $a_0 \in S$ then by (7) $\Aut(\Gamma)_1$ fixes $a_0$, and we are done, while if $a', a_1a' \in S$ then by (9) $\Aut(\Gamma)_1$ fixes every vertex of $S$.  In the second case, by (8) $\Aut(\Gamma)_1$ also fixes $a_\ell$, so by Lemma~\ref{Watkins-Nowitz} it fixes $a_\ell a_{\ell-1}^{-1}$.

If $k=\ell-1$, then every directed induced path of length $k$ has either $a_1, \ldots, a_{\ell-2}$, $a_2, \ldots, a_{\ell-1}$ or $a_3, \ldots, a_\ell$ as its interior vertices. In the first case, the path must begin with $a_2a_1^{-1}$, and (using (3) and (5)) $a_2a_1^{-1}, a_1, \ldots, a_\ell$ is a longer induced directed path, contradicting our assumption. In the last case, the path must end with $a_\ell a_{\ell-1}^{-1}$, and (using (3) and (5)) $a_1, \ldots, a_\ell, a_{\ell}a_{\ell-1}^{-1}$ is a longer induced directed path, contradicting our assumption. So every directed induced path of length $k$ has $a_2, \ldots, a_{\ell-1}$ as its interior vertices. Thus, $\Aut(\Gamma)_1$ fixes all of the vertices $a_2, \ldots, a_{\ell-1}$. By (8), $\Aut(\Gamma)_1$ also fixes $a_\ell$ (and by Lemma~\ref{Watkins-Nowitz} the vertex $a_\ell a_{\ell-1}^{-1}$ is now also fixed by $\Aut(\Gamma)_1$). This means that the only vertices of $S$ that $\Aut(\Gamma)_1$ can move are $a_1$, $a_2a_1^{-1}$, $a_0$, $a_1a'$, and $a'$. Recall that if $a_0 \in S$ then $a_1a',a' \not\in S$ and vice versa.
If $a_0 \in S$ then by (7) $\Aut(\Gamma)_1$ fixes $a_0$, and we are done, while if $a', a_1a' \in S$ then by (9) $\Aut(\Gamma)_1$ fixes every vertex of $S$.  
\end{proof}

Using the above facts, we can show that for three particular ways to define Cayley graphs on a group that admits the type of ordered generating set we are looking for, at least one of the ways always produces an ORR for $G$.

\begin{lemma}\label{Gamma-ORR}
Let $G$ be a  group that admits a five-product-avoiding generating set $T=\{b_1, \ldots, b_\ell\}$ with $\ell \ge 4$ and with the following properties:
\begin{enumerate}
\item[(i)] for every $t \in T$, $|t|>2$; and
\item[(ii)] for $1 \le i \le \ell-1$, $|b_{i+1}b_i^{-1}|>2$.
\end{enumerate}
Let $X=\{b_{i+1}b_i^{-1}\mid 1 \le i \le \ell-1\}$ and let $S=T \cup X$. If $[b_1,b_2]=1$ and $|b_1b_2|>2$, then let $S'=S \cup \{b_0=b_1b_2\}$; while if $[b_1,b_2]=1$ and  $|b_1b_2|=2$, then let $S''=(S\setminus \{b_1\})\cup \{b_1^{-1}\}$. Let $\Gamma=\Cay(G,S)$, $\Gamma'=\Cay(G,S')$, and $\Gamma''=\Cay(G,S'')$. Then
\begin{enumerate}
\item $\Gamma$ is an ORR for $G$ unless 
$[b_1,b_2]=1$; 
\item if $[b_1,b_2]=1$ and $|b_1b_2|>2$ then $\Gamma'$ is an ORR for $G$; and
\item if $[b_1,b_2]=1$ and $|b_1b_2|=2$ then $\Gamma''$ is an ORR for $G$.
\end{enumerate} 
\end{lemma}

\begin{proof}[Proof of~$(1)$] Observe that $\Gamma$ satisfies the hypotheses of Proposition~\ref{usefulfacts}. By Proposition~\ref{usefulfacts}~(1) and~(2), $\Gamma$ is an oriented Cayley digraph.
So, by Proposition~\ref{usefulfacts}~(10), if $\Aut(\Gamma)_1$ is non-trivial, it must fix every vertex but $b_1$ and $b_2b_1^{-1}$, which must lie in an orbit of length $2$. In particular, $\Aut(\Gamma)_1$ fixes $b_2$. By Proposition~\ref{usefulfacts}~(5), there is an arc from $b_1$ to $b_2$, so if $\Aut(\Gamma)_1$ is non-trivial there must also be an arc from $b_2b_1^{-1}$ to $b_2$. By Proposition~\ref{usefulfacts}~(4), this implies that $[b_1,b_2]=1$, the exception listed. If $\Aut(\Gamma)_1$ is trivial, then by Lemma~\ref{Watkins-Nowitz}, $\Gamma$ is an ORR for $G$.
\end{proof}

\begin{proof}[Proof of~$(2)$] Observe that $\Gamma'$ satisfies the hypotheses of Proposition~\ref{usefulfacts}. By Proposition~\ref{usefulfacts}~(1) and~(2), $\Gamma'$ is an oriented Cayley digraph.
So, by Proposition~\ref{usefulfacts}~(10), if $\Aut(\Gamma')_1$ is non-trivial, it must fix every vertex but $b_1$ and $b_2b_1^{-1}$, which must lie in an orbit of length $2$. In particular, $\Aut(\Gamma')_1$ fixes $b_0=b_1b_2$. Since there is an arc from $b_1$ to $b_1b_2$, if $\Aut(\Gamma)_1$ is non-trivial there must also be an arc from $b_2b_1^{-1}$ to $b_0$. This cannot occur, by Proposition~\ref{usefulfacts}~(3). Thus $\Aut(\Gamma')_1$ is trivial, and by Lemma~\ref{Watkins-Nowitz} $\Gamma'$ is an ORR for $G$.
 \end{proof}

\begin{proof}[Proof of~$(3)$] Observe that $\Gamma''$ satisfies the hypotheses of Proposition~\ref{usefulfacts} with $a_i=b_{i+1}$ for $1 \le i \le \ell-1$, and $a'=b_1^{-1}$ (so $a_1a'=b_2b_1^{-1}$). By Proposition~\ref{usefulfacts}~(1) and~(2), $\Gamma''$ is an oriented Cayley digraph.
So, by Proposition~\ref{usefulfacts}~(9) and~(10), $\Aut(\Gamma'')_1$ is trivial and by Lemma~\ref{Watkins-Nowitz} $\Gamma''$ is an ORR for $G$.
\end{proof}

We conclude with our main result for this section, which is essentially a combination of the preceding results together with some material from preceding sections.

\begin{theorem}\label{JoysLemma}
Let $G$ be a finite group that admits a five-product-avoiding generating set $T=\{a_1, \ldots, a_\ell\}$ with the following properties:
\begin{enumerate}
\item[(i)] for every $t \in T$, $|t|>2$; and
\item[(ii)] for every $i\in \{1,\ldots,\ell-1\}$, $|a_{i+1}a_i^{-1}|>2$.
\end{enumerate}
Then $G$ admits an ORR if and only if $G \not\cong Q_8$, $G \not\cong C_3 \times C_2^3$, and $G \not\cong C_3 \times C_3$.
\end{theorem}

\begin{proof}
If $d(G)=1$, then $\Cay(G,\{a_1\})$ is an ORR for $G$, except when $|a_1|=2$. However, when $|a_1|=2$, $\Cay(G,\emptyset)$ is an ORR.

If $d(G) = 2$, then by Theorem~\ref{thrm:A1}, $G$ admits an ORR if and only if $G \not\cong C_3 \times C_3$, $G \not\cong Q_8$, $G \not\cong C_4 \times C_2$, and $G\not\cong \langle a,b \mid a^4=b^4=(ab)^2=(ab^{-1})^2=1\rangle$. The first two groups in this list are listed in our statement as exceptions; the other two groups do not admit generators $a,b$ with $|a|, |b|, |ba^{-1}|>2$.

If $d(G)=3$, then by Theorem~\ref{thrm:A2}, $G$ admits an ORR if and only if $G \not\cong C_3 \times C_2^3$, $G \not\cong C_4 \times C_2^2$, and $G$ is not isomorphic to either of the two other groups listed in the statement of that theorem. The first of these groups is listed in our statement as an exception. None of the other three groups admits a generating set $\{a,b,c\}$ with $|a|, |b|, |c|, |ba^{-1}|, |cb^{-1}|>2$.

If $d(G)\ge 4$, then $\ell \ge 4$ and $G$ admits an ORR as an immediate consequence of Lemma~\ref{Gamma-ORR}.
\end{proof}

\section{Non-solvable groups}\label{non-sol}


The lemma with which we begin this section will allow us to apply the theorem from the previous section to any  group that has a unique minimal normal subgroup, if that subgroup is non-abelian.

\begin{lemma}\label{lemma:A4}Let $G$ be a   group with a unique minimal normal subgroup, $N$. Suppose that $N$ is non-abelian. Then $G$ admits an irredundant generating set $\{a_1,\,\ldots,a_d\}$ with $d\ge 2$,  and $$|a_1|,\ldots,|a_d|,|a_2a_{1}^{-1}|,\ldots,|a_d a_{d-1}^{-1}|>2.$$
\end{lemma}
\begin{proof}
Note that $G$ cannot be generalized dihedral (since these groups have non-identity normal abelian subgroups) or abelian.

Observe that $d(G)>1$ because $G$ is not cyclic. If $d(G)=2$, then by Lemma~\ref{lemma:A2} and Lemma~\ref{lemma:A3} the group $G$ admits a generating set $\{a_1,a_2\}$ with $|a_1|$, $|a_2|$, $|a_2a_1^{-1}|>2$
, because $G$ cannot have a non-identity abelian normal subgroup. Suppose $d(G)=3$; here we use the subdivision in Lemma~\ref{lemma:2A}. We have already noted that part~(i) and~(ii) in Lemma~\ref{lemma:2A} cannot arise. Assume that $G$ admits a generating set $\{a,b,c\}$ with 
$$|a|,|b|,|c|>2,\,(ab)^2=(ab^{-1})^2=(bc)^2=(bc^{-1})^2=1,$$
that is, $G$ admits a generating set as in Lemma~\ref{lemma:2A}~(iii),~(iv) or~(v). From Lemma~\ref{lemma:A-1}, we deduce that $(b^2)^a=b^{-2}$ and $(b^2)^c=b^{-2}$ and hence $\langle b^2\rangle\unlhd \langle a,b,c\rangle=G$. Since $G$ has no non-identity abelian normal subgroups, we obtain $|b|=2$, contradicting $|b|>2$. Therefore $G$ admits a generating set as in Lemma~\ref{lemma:2A}~(vi) and this lemma holds.

We may now assume that $$d(G)>3.$$ 
By minimality, we have $$N=T_1\times T_2\times \cdots\times T_\kappa,$$ where $\kappa\in\mathbb{N}\setminus\{0\}$,  $T_1,\ldots,T_\kappa$ are pairwise isomorphic non-abelian simple groups and $G$ acts transitively by conjugation on the set $\{T_1,\ldots,T_\kappa\}$.
From~\cite[Corollary]{DVL}, we have $d(G)\in \{2,3\}$ when $\kappa=1$. As $d(G)>3$, we have $$\kappa>1,$$ that is, $N$ is not a non-abelian simple group. Set $$\ell:=d(G/N).$$
We now prove some preliminary claims.

\smallskip

\noindent\textsc{Claim 0. }There exist $x,y\in T_1$ with $T_1=\langle x,y\rangle$, and $|x|,|y|,|yx^{-1}|>2$.

\smallskip 

\noindent Since  $T_1$ is a non-abelian simple group, we have $d(T_1)=2$ (see for instance~\cite[Corollary]{DVL}; this uses the Classification of Finite Simple Groups).  Now, by Lemmas~\ref{lemma:A2} and~\ref{lemma:A3} there exist $x,y\in T_1$ with $T_1=\langle x,y\rangle$ and $|x|, |y| ,|yx^{-1}|>2$ because $T_1$ has no non-identity abelian normal subgroups.~$_\blacksquare$

\smallskip

\noindent\textsc{Claim 1. }Let $a_1,\ldots,a_\ell$ be elements of $G$ such that $G=\langle a_1,\ldots,a_\ell,N\rangle$, and removing any $a_i$ from this set would generate a proper subgroup of $G$. Then there exists $\iota\in\{0,1,2\}$ and   $n_1,\ldots,n_\iota\in T_1$ such that
\begin{enumerate}
\item[(i)]$\{a_1,\ldots,a_\ell,n_1,\ldots,n_\iota\}$ is an irredundant generating set for $G$,
\item[(ii)]$|n_1|,\ldots,|n_\iota|>2$, and if $\iota=2$ then $|n_2 n_{1}^{-1}|>2$.
\end{enumerate}
(Observe that $\iota$ and $n_1,\ldots,n_\iota$ depend upon $a_1,\ldots,a_\ell$.)

\smallskip

\noindent Let $x,y\in T_1$ with $T_1=\langle x,y\rangle$ and $|x|,|y|,|yx^{-1}|>2$ as in Claim~0. The group $G/N$ acts transitively by conjugation on $\{T_1,\ldots,T_\kappa\}$ and hence $$\langle x^h,y^h\mid h\in \langle a_1,\ldots,a_\ell\rangle\rangle=T_1\times\cdots\times T_\kappa=N.$$ Therefore $G=\langle a_1,\ldots,a_\ell,x,y\rangle$ and (replacing $x$ by $y$ if necessary) we have three possibilities:
\[G=
\begin{cases}
\langle a_1,\ldots,a_\ell\rangle,\\
\langle a_1,\ldots,a_\ell,x\rangle,\\
\langle a_1,\ldots,a_\ell,x,y\rangle.\\
\end{cases}
\]
In the first case define $\iota:=0$, in the second case define $\iota:=1$ and $n_1:=x$, in the third case define $\iota:=2$, $n_1:=x$ and $n_2:=y$. This definition implies that $\{a_1,\ldots,a_\ell,n_1,\ldots,n_\iota\}$ is an irredundant generating set for $G$. Our choice of $x$ and $y$ immediately gives condition~(ii).

\smallskip

\noindent\textsc{Claim 2. }For every $g\in G\setminus N$, there exists $n_g\in N$ with $|gn_g|>2$.
\smallskip

\noindent If $|gn|\leq 2$ for every $n\in N$, then $|gn|=2$ for every $n\in N$ because $g\notin N$. In particular, $g^2=1$. Moreover, $1=(gn)^2=gngn=g^2n^gn=n^gn$ and hence $n^g=n^{-1}$, for every $n\in  N$. Thus $g$ acts by conjugation on $N$ inverting each of its elements and hence $N$ is abelian, a contradiction.~$_\blacksquare$

\smallskip

\noindent\textsc{Claim 3. }Let $a\in G\setminus N$ with $|a|>2$. For every $y\in G\setminus N$, there exists $n_y\in N$ such that, for $b:=yn_y$, we have:
$$|b|>2,\text{ and }|ba^{-1}|>2 \text{ (as well as }|a|>2).$$

\smallskip

\noindent Let $y\in G\setminus N$ and consider the following sets:
\begin{align*}
\mathcal{S}_1&:=\{n\in N\mid|yn|= 2\},\\
\mathcal{S}_2&:=\{n\in N\mid|yna^{-1}|\leq 2\}.
\end{align*}

We show that $|\mathcal{S}_1|\leq |N|/4$. If $\mathcal{S}_1=\emptyset$, then there is nothing to prove; thus we may suppose that $\mathcal{S}_1\neq \emptyset$ and hence there exists $n_0\in N$ with $|yn_0|= 2$. Define $y':=yn_0$. For $n\in \mathcal{S}_1$, define $n':=n_0^{-1}n$. Now, given $n\in\mathcal{S}_1$, we obtain $$1=(yn)^2=(y'n')^2=y'n'y'n'=y'^2(n')^{y'}n'=(n')^{y'}n'$$ and hence $(n')^{y'}=(n')^{-1}$. This shows that the element $y'$ acts by conjugation on $N$ as an automorphism  inverting $n'$. Now, as $N\not\cong \Alt(5)$ because $\kappa>1$, from a result of Potter~\cite[Theorem 3.1]{Potter}, we see that an automorphism of the non-solvable group $N$ can invert at most $|N|/4$ of its elements and hence we have at most $|N|/4$ choices for $n'$. Therefore $|\mathcal{S}_1|\leq |N|/4$. 

Following the thread of the previous paragraph, we show that $|\mathcal{S}_2|\leq |N|/4$. If $\mathcal{S}_2=\emptyset$, then there is nothing to prove; thus  we may suppose that $\mathcal{S}_2\neq \emptyset$ and hence there exists $n_0\in N$ with $|yn_0a^{-1}|\le 2$. Define $y':=yn_0$. For $n\in \mathcal{S}_2$, define $n':=n_0^{-1}n$. Observe that $y'a^{-1}=yn_0a^{-1}$ has order at most $2$ because $n_0\in\mathcal{S}_2$. Now, given $n\in \mathcal{S}_2$, we obtain
\begin{eqnarray*}
1&=&(yna^{-1})^2=(y'n'a^{-1})^2=y'n'a^{-1}y'n'a^{-1}=y'a^{-1}(n')^{a^{-1}}y'n'a^{-1}\\
&=&y'a^{-1}(n')^{a^{-1}}y'a^{-1}(n')^{a^{-1}}=((n')^{a^{-1}})^{y'a^{-1}}(n')^{a^{-1}}
\end{eqnarray*} and hence $((n')^{a^{-1}})^{y'a^{-1}}=((n')^{a^{-1}})^{-1}$. This shows that the element $y'a^{-1}$ acts by conjugation on $N$  as an automorphism inverting $(n')^{a^{-1}}$. Now, as $N\not\cong \Alt(5)$ because $\kappa>1$, from~\cite[Theorem 3.1]{Potter}, we see that an automorphism of the non-solvable group $N$ can invert at most $|N|/4$ of its elements and hence we have at most $|N|/4$ choices for $n'$. Therefore $|\mathcal{S}_2|\leq |N|/4$.


Summing up
$$|\mathcal{S}_1\cup\mathcal{S}_2
|\leq 2\frac{|N|}{4}
=\frac{1}{2}|N|<|N|.$$
In particular, there exists $\bar{n}\in N$ with $\bar{n}\notin\mathcal{S}_1\cup\mathcal{S}_2$. 
  Now the claim follows by taking $b:=y\bar{n}$.~$_\blacksquare$

\smallskip

\noindent\textsc{Claim 4. }For every $a_1,\ldots,a_\ell\in G\setminus N$ with $G/N=\langle a_1,\ldots,a_\ell,N\rangle$, there exist $n_1,\ldots,n_\ell\in N$ such that the elements $a_1':=a_1n_1,\ldots,a_\ell':=a_\ell n_\ell$ satisfy 
\begin{enumerate}
\item[(i)]$|a_i'|>2$ for every $i\in \{1,\ldots,\ell\}$,
\item[(ii)]$|a_{i+1}'a_i'^{-1}|>2$ for every $i\in \{1,\ldots,\ell-1\}$.
\end{enumerate}

\smallskip

\noindent We argue by induction on $\ell$. From Claim~$2$ there exists $n_1\in N$ such that $a_1':=a_1n_1$ has order greater than $2$. Now, from Claim~$3$, there exists $n_2\in N$ such that, by setting $a_2':=a_2n_2$, we have $|a_2'|>2$, and $|a_2'a_1'^{-1}|>2$.   
We may now use Claim~$3$ iteratively  (first with $a:=a_2'$ and $y:=a_3$) to construct $a_3'$ with $|a_3'|>2$ and $|a_3'a_2'^{-1}|>2$, etc.~$_\blacksquare$

\smallskip

Let $a_1,\ldots,a_\ell\in G$ with $G=\langle a_1,\ldots,a_\ell,N\rangle$. 
Since $T_1$ is normal in $N$ but not in $G$ (as $\kappa>1$), some generator of $G$ must fail to normalise $T_1$. Relabeling the index set $\{1,\ldots,\ell\}$ if necessary, we may assume that
\begin{equation}\label{kappa}
T_1^{a_\ell}\neq T_1.
\end{equation}
Replacing the minimal generating set $\{a_1,\ldots,a_\ell\}$ of $G/N$ if necessary, we may suppose that $a_1,\ldots,a_\ell$ satisfy also the conditions in~(i) and~(ii) of Claim~4. Observe that~\eqref{kappa} is still satisfied. Now, let $\iota\in\{0,1,2\}$ and $n_1,\ldots,n_\iota\in T_1$ be as in Claim~1. 

If $\iota=0$, then  $\{a_1,\ldots,a_\ell\}$ is a minimal (and hence irredundant) generating set of $G$ satisfying the conclusion of this lemma. We may then assume that $\iota>0$.

\smallskip

\noindent\textsc{Claim 5. }There exists $\varepsilon\in \{1,-1\}$ such that $|a_{\ell}n_1^{-\varepsilon}|>2$ (and hence $|n_1^\varepsilon a_\ell^{-1}|>2$).

\smallskip

\noindent We argue by contradiction and we assume that $|a_\ell n_1|=|a_\ell n_1^{-1}|=2$. Now Lemma~\ref{lemma:A-1} yields $(n_1^2)^{a_\ell}=n_1^2$ and hence $a_\ell$ centralizes a non-identity element of $T_1$,  contradicting $T_1^{a_\ell}\ne T_1$.~$_\blacksquare$ 

\smallskip

In view of Claim~5, there exists $\varepsilon\in \{1,-1\}$ with $|a_\ell n_1^{-\varepsilon}|>2$. Observe that if $\varepsilon=-1$ and $|n_2^{-1}n_1|=2$ then $n_1n_2^{-1}n_1=n_2$ so $|n_1n_2^{-1}|=|n_2n_1^{-1}|=2$, a contradiction. Now it is immediate to check (using the way that $a_1,\ldots,a_\ell$ and $n_1,\ldots,n_\iota$ were defined) that
$$\{a_1,a_2,\ldots,a_\ell,n_1^\varepsilon,\ldots,n_\iota^\varepsilon\}$$
is an irredundant generating set for $G$ and satisfies the conditions of this lemma.
\end{proof}

\begin{theorem}\label{mainthm}Let $G$ be a finite non-solvable group. Then $G$ admits an irredundant generating set $\{a_1,\ldots,a_d\}$ with $$|a_1|,\ldots,|a_d|,|a_2a_1^{-1}|,\ldots,|a_d a_{d-1}^{-1}|>2.$$
In particular, $G$ admits an ORR.
\end{theorem}
\begin{proof}
We first prove the existence of the required irredundant generating set.
We argue by contradiction and among all  non-solvable groups witnessing the incorrectness of this theorem, choose $G$ with $|G|$ as small as possible. Let $K$ be a minimal normal subgroup of $G$. 

Assume that $G/K$ is non-solvable. By the minimality of $|G|$, $G/K$ admits an irredundant generating set $\{a_1'K,\ldots,a_\ell' K\}$ with
$$|a_1'K|,\ldots,|a_\ell' K|,|a_2'a_1'^{-1}K|,\ldots,|a_\ell' a_{\ell-1}'^{-1}K|>2.$$
Observe that $\ell>1$ because the non-solvable group $G/K$ cannot be cyclic. 

Choose $k_1,\ldots,k_\ell\in K$ such that the number $\iota\in\mathbb{N}$ of elements $x_1,\ldots,x_\iota\in K$ necessary to have $$G=\langle a_1'k_1,\ldots,a_\ell' k_\ell,x_1,\ldots,x_\iota\rangle$$ is minimum. 

Define $a_i:=a_i'k_i$ for every $i\in \{1,\ldots,\ell\}$. Let $x_1,\ldots,x_\iota\in K$ with $G=\langle a_1,\ldots,a_\ell,x_1,\ldots,x_\iota\rangle$.
For each $i\in \{1,\ldots,\iota\}$, let
$$a_{\ell+i}=\begin{cases}
a_{\ell-1}x_i&\textrm{if }i \textrm{ is odd},\\
a_{\ell}x_i&\textrm{if }i \textrm{ is even}.
\end{cases}$$

Now, $\{a_1,\ldots,a_{\ell+\iota}\}$ is a generating set for $G$ by construction. Moreover, for $j\in\{1,\ldots,\ell+\iota\}$, we have \[
|a_j|\ge\begin{cases}
|a_jK|=|a_j'K|>2&\textrm{if }j\in \{1,\ldots,\ell\},\\
|a_jK|=|a_{\ell-1}'K|>2&\textrm{if }j\in \{\ell+1,\ldots,\ell+\iota\} \textrm{ and }j-\ell \textrm{ is odd},\\
|a_jK|=|a_\ell'K|>2&\textrm{if }j\in \{\ell+1,\ldots,\ell+\iota\} \textrm{ and }j-\ell \textrm{ is even};
\end{cases}
\]
furthermore, for $j\in \{1,\ldots,\ell+\iota-1\}$, we have
\[
|a_{j+1}a_j^{-1}|\ge\begin{cases}
|a_{j+1}a_{j}^{-1}K|=|a_{j+1}'a_j'^{-1}K|>2&\textrm{if }j\in \{1,\ldots,\ell-1\},\\
|a_{\ell}'a_{\ell-1}'^{-1}K|>2&\textrm{if }j\in \{\ell,\ell+1,\ldots,\ell+\iota-1\}.
\end{cases}
\]

It remains to prove that $L=\{a_1,\ldots,a_{\ell+\iota}\}$ is an irredundant generating set. We argue by contradiction and suppose that it is not irredundant. Since $\{a_1K,\ldots,a_\ell K\}$ is an irredundant generating set for $G/K$, we see that we cannot delete any of $a_1,\ldots,a_{\ell-2}$ from $L$  and still have a generating set for $G$. Suppose that by removing $a_{\ell+i}$ from $L$, where $i\in \{1,\ldots,\iota\}$, we still have a generating set for $G$. Recalling that $a_{\ell+j}\in \{a_{\ell-1} x_j,a_{\ell}x_j\}$, we get
\begin{eqnarray*}
G&=&\langle a_1,\ldots,a_{\ell+\iota}\rangle=\langle a_1,\ldots,a_\ell,a_{\ell+1},\ldots,a_{\ell+i-1},a_{\ell+i+1},\ldots,a_{\ell+\iota}\rangle\\
&=&\langle a_1,\ldots,a_\ell,x_1,\ldots,x_{i-1},x_{i+1},\ldots,x_{\iota}\rangle,
\end{eqnarray*}
contradicting the minimality of $\iota$. Suppose that by removing $a_{\ell-1}$ from $L$,  we still have a generating set for $G$. Recalling that $a_{\ell+j}\in \{a_{\ell-1} x_j,a_{\ell}x_j\}$ and $a_{\ell+1}=a_{\ell-1}x_1$, we get
\begin{eqnarray*}
G&=&\langle a_1,\ldots,a_{\ell+\iota}\rangle=\langle a_1,\ldots,a_{\ell-2},a_{\ell},a_{\ell+1},\ldots,a_{\ell+\iota}\rangle\\
&=&\langle a_1,\ldots,a_{\ell-2},a_{\ell},\,a_{\ell-1}x_1,a_\ell x_2,a_{\ell-1}x_3,a_\ell x_4,a_{\ell-1}x_5,a_\ell x_6,\ldots\rangle\\
&=&\langle a_1,\ldots,a_{\ell-2},a_{\ell},\,a_{\ell-1}x_1, x_2,a_{\ell-1}x_3, x_4,a_{\ell-1}x_5, x_6,\ldots\rangle\\
&=&\langle a_1,\ldots,a_{\ell-2},a_{\ell},\,a_{\ell-1}x_1, x_2,x_1^{-1}x_3, x_4,x_1^{-1}x_5, x_6,\ldots\rangle\\
&=&\langle a_1,\ldots,a_{\ell-2},a_{\ell-1}x_1,a_{\ell},\, x_2,x_1^{-1}x_3, x_4,x_1^{-1}x_5, x_6,\ldots\rangle\\
\end{eqnarray*}
In particular,  to obtain a generating set for $G$ we need to add only $\iota-1$ elements of $K$ to $\{a_1,\ldots,a_{\ell-2},a_{\ell-1}x_1,a_\ell\}$,
contradicting again the minimality of $\iota$. An entirely similar argument shows that by removing $a_{\ell}$  we no longer have a  generating set for $G$. This concludes the proof when $G/K$ is non-solvable.

\smallskip

From the above, we may assume that $G/K$ is solvable for every minimal normal subgroup $K$ of $G$. Suppose that $G$ has two distinct minimal normal subgroups, say $N_1$ and $N_2$. In particular, $G/N_1$ and $G/N_2$ are solvable, and hence so is $G$ because $G$ embeds into $G/N_1\times G/N_2$. Therefore $G$ has  a unique minimal normal subgroup, say $N$. As $G/N$ is solvable, $N$ must be non-abelian and hence Lemma~\ref{lemma:A4} shows the existence of the required generating set. 

\smallskip 

Since $G$ is non-solvable, we cannot have $G \cong Q_8, C_3\times C_2^3$, or $C_3 \times C_3$. Therefore $G$ admits an ORR from Theorem~\ref{JoysLemma}.
\end{proof}

\section{Acknowledgments}
The authors would like to thank the anonymous referees whose helpful reports improved this paper.

This research was supported in part by the National Science
  and Engineering Research Council of Canada, Discovery Grant 238552-2011.

\end{document}